\documentclass[10pt]{article}
\usepackage[a4paper,margin=2.5cm]{geometry}

\usepackage{graphicx}
\usepackage{amsmath}
\usepackage{amssymb}
\usepackage{amsfonts}
\usepackage{amsthm}
\usepackage{ifthen}
\usepackage[pdfusetitle]{hyperref}
\usepackage{mathtools}
\usepackage{authblk}
\newcommand{\fa}{\hbox{ for all }}
\newcommand{\ns}{\calh_k(\Omega)}
\newcommand{\R}{\mathbb{R}}
\newcommand{\N}{\mathbb{N}}
\newcommand{\calh}{\mathcal H}
\newcommand{\inner}[2]{\ifthenelse{\equal{#2}{}}{\left\langle\cdot,\cdot\right\rangle_{#1}}{\left\langle#2\right\rangle_{#1}}}
\newcommand{\norm}[2]{\ifthenelse{\equal{#2}{}}{\left\|\cdot\right\|_{#1}}{\left\|#2\right\|_{#1}}}
\newcommand{\seminorm}[2]{\ifthenelse{\equal{#2}{}}{\left|\cdot\right|_{#1}}{\left|#2\right|_{#1}}}
\DeclareMathOperator{\Sp}{span}
\DeclareMathOperator{\argmax}{arg\max}
\DeclareMathOperator{\argmin}{arg\min}
\DeclareMathOperator{\conv}{co}
\DeclareMathOperator{\clos}{clos}
\DeclareMathOperator{\Lip}{Lip}
\newcommand{\calb}{\mathcal B}

\newtheorem{theorem}{Theorem}

\newtheorem{proposition}[theorem]{Proposition}
\newtheorem{lemma}[theorem]{Lemma}
\newtheorem{remark}[theorem]{Remark}
\newtheorem{corollary}[theorem]{Corollary}

\graphicspath{{./}}

\title{On the optimality of target-data-dependent kernel greedy interpolation in Sobolev Reproducing Kernel 
Hilbert Spaces}

\author[1]{Gabriele Santin
\thanks{gsantin@fbk.eu}}
\affil[1]{Digital Society Center, Bruno Kessler Foundation (Trento, Italy)}

\author[2]{Tizian Wenzel
\thanks{tizian.wenzel@mathematik.uni-stuttgart.de}}
\author{Bernard Haasdonk
\thanks{bernard.haasdonk@mathematik.uni-stuttgart.de}}
\affil[2]{Institute for Applied Analysis and Numerical Simulation, University of Stuttgart (Germany)}

\date{\today}
\begin{document}
\maketitle

\begin{abstract}
Kernel interpolation is a versatile tool for the approximation of functions from data, and it can be proven to have some optimality properties when used with 
kernels related to certain Sobolev spaces.

In the context of interpolation, the selection of optimal function sampling locations is a central problem, both from a practical perspective, and as an 
interesting theoretical question.
Greedy interpolation algorithms provide a viable solution for this task, being efficient to run and provably accurate in their approximation.

In this paper we close a gap that is present in the convergence theory for these algorithms by employing a recent result on 
general greedy algorithms. 
This modification leads to new convergence rates which match the optimal ones when restricted to the $P$-greedy target-data-independent selection rule, and 
can additionally be proven to be optimal when they fully exploit adaptivity ($f$-greedy).

Other than closing this gap, the new results have some significance in the broader setting of the optimality of general approximation algorithms in 
Reproducing Kernel Hilbert Spaces, as they allow us to compare adaptive interpolation with non-adaptive best nonlinear approximation.
\end{abstract}

\section{Interpolation and greedy algorithms}

We denote a strictly positive definite and continuous kernel as $k$, and let $\ns$ be its Reproducing Kernel Hilbert Space (or native space) over a bounded set 
$\Omega\subset\R^d$. The 
properties of these kernels and spaces, and their approximation theory, are well studied e.g. in \cite{Fasshauer2015,Saitoh2016,Wendland2005}, to which we 
refer 
the interested reader.

For a function $f:\Omega\to\R$ and finitely many pairwise distinct points $X\subset \Omega$, there is a unique function $I_X f \in V(X)\coloneqq
\Sp\{k(\cdot, x):x\in 
X\}$ which interpolates $f$ at $X$. This kernel interpolant is the interpolant of minimal norm in $\ns$, and it coincides with the orthogonal projection 
$\Pi_{V(X)} f$ of $f$ onto $V(X)$ whenever $f\in \ns$.

We will make repeated use of the power function $P_X$ associated to this interpolation process. For $X\subset \Omega$, $x\in\Omega$, we set
\begin{equation}\label{eq:power_as_proj}
P_{X}(x)
\coloneqq
\sup\limits_{0\neq f\in \ns}\frac{\left|f(x)- I_X f(x)\right|}{\norm{\ns}{f}}
= \norm{\ns}{k(\cdot, x) - \Pi_{V(X)} k(\cdot, x)},
\end{equation}
which clearly provides a worst-case error bound
\begin{equation}\label{eq:power_fun_err_bound}
\left|f(x)- I_X f(x)\right|\leq P_X(x) \norm{\ns}{f}\;\;\fa f\in\ns,\ x\in \Omega.
\end{equation}
It can easily be checked that $P_{Y}(x)\leq P_X(x)$ for $X\subset Y$ and $x\in \Omega$, and $P_\emptyset(x) = \sqrt{k(x,x)}$.
With this in mind, we assume throughout the paper that $k(x,x)\leq 1$ for all $x\in\Omega$, which implies
\begin{equation}\label{eq:kernel_normalization}
\norm{L_\infty(\Omega)}{P_X} \leq \sup_{x\in\Omega}\sqrt{k(x,x)}\leq 1 \;\;\fa X\subset \Omega.
\end{equation}
Observe that for a bounded kernel this condition can be easily met by scaling, without changing the corresponding native space but only rescaling its norm.

Whenever $k$ is well defined on a bounded set $\Omega'\supset \Omega$ (i.e., it can be evaluated and it is continuous and strictly positive definite on 
$\Omega'$), it additionally holds
\begin{equation}\label{eq:rkhs_extension}
\ns\subset \calh_k(\Omega'),
\end{equation}
and this inclusion is in fact a continuous embedding with
\begin{equation}\label{eq:rkhs_embedding}
\norm{\calh_k(\Omega')}{E f}= \norm{\ns}{f}\;\;\fa f\in \ns,
\end{equation}
where $E f\in \calh_k(\Omega')$ is the extension of $f\in\ns$ to $\calh_k(\Omega')$ (see Theorem 10.46 in \cite{Wendland2005}). For simplicity of notation, we 
will simply identify $f$ in place of $E f$ in the following.

We will work with incrementally selected sets of points. For an arbitrary enumeration $\{x_1, \dots, x_N\}$ of $X\subset \Omega$, setting $X_i\coloneqq\{x_j, 
1\leq j\leq i\}$ there is an identity
\begin{equation}\label{eq:norm_as_expansion}
\norm{\ns}{I_{X} f}^2 = \sum_{i=1}^N \frac{\left|(f - I_{X_{i-1}} f)(x_i)\right|}{P_{X_{i-1}}(x_i)}^2,\;\;f\in\ns,
\end{equation}
which can be proven using the Newton basis of $V(X)$ \cite{Muller2009,Pazouki2011}.

In this paper we will prove refined rates of convergence for some kernel greedy interpolation algorithms. These algorithms amount at computing kernel 
interpolants on nested sets of points $\emptyset\eqqcolon X_0\subset X_1\subset\dots$, where $X_{n+1}\coloneqq X_{n}\cup\{x_{n+1}\}$ and $x_{n+1}\in\Omega$ is 
selected 
at each iteration according to a certain selection rule. 
The results that we are going to present apply to the scale of $\beta$-greedy algorithms \cite{Wenzel2022c}, $\beta\geq 0$, which are defined as
\begin{equation*}
x_{n+1}\coloneqq \argmax\limits_{x\in \Omega\setminus X_n}\ \eta_{\beta}(f, X_n)(x),
\end{equation*}
with the selection rule
\begin{equation}\label{eq:beta_greedy}
\eta_{\beta}(f, X_n)(x)\coloneqq\left|\left(f - I_{X_n} f\right)(x)\right|^\beta P_{X_n}(x)^{1-\beta},
\end{equation}
which is extended by continuity to $\beta=\infty$ as
\begin{equation*}
\eta_{\infty}(f, X_n)(x)\coloneqq \frac{\left|\left(f - I_{X_n} f\right)(x)\right|}{P_{X_n}(x)}.
\end{equation*}

Greedy algorithms are ubiquitous in applied mathematics, and in particular in approximation theory \cite{Temlyakov2008}, and to the best of our knowledge 
the first kernel greedy interpolation algorithm has been studied in \cite{SchWen2000}. This is the $f$-greedy algorithm, which is the $\beta$-greedy algorithm 
with $\beta=1$. Other significant instances are the $P$-greedy \cite{DeMarchi2005} ($\beta=0$), $f\cdot P$-greedy \cite{Dutta2021a} ($\beta=1/2$, originally named \emph{psr}-greedy), and
$f/P$-greedy \cite{Mueller2009} ($\beta=\infty$) algorithms.
Section 2.2 in \cite{Wenzel2022c} contains a detailed account of the properties of these different algorithms, and an overview of the progressive 
development of their theoretical understanding.

The current most comprehensive convergence theory for these algorithms has been presented so far in \cite{Santin2017x,Wenzel2021a} for $\beta=0$, and in 
\cite{Wenzel2022c} for general $\beta$.
Specifically, Theorem 4.1 in \cite{Santin2017x,Wenzel2021a} proves that any rate (algebraic or exponential) of uniform best interpolation in $\ns$ transfers
to exactly the same rate, with different constants, for the $P$-greedy algorithm. In particular, $P$-greedy achieves the worst-case optimal interpolation error 
in $\ns$.
This result has been significantly extended to the entire scale of $\beta$-greedy algorithms in \cite{Wenzel2022c}, where it has been proven that the greedy 
algorithms $\beta\in[0,1]$ leads to an additional $n^{-\beta/2}$ factor in the interpolation error bounds.
This result proved for the first time that the barrier for interpolation with optimal points can be broken with adaptivity, and with a rate factor that is kernel- and
dimension-independent.

These results are of particular interest for kernels whose native space is norm equivalent to the Sobolev space $W_2^\tau(\Omega)$. 
This is in particular the case for translational invariant kernels $k(x,y)\coloneqq \phi(x-y)$, $\phi:\R^d\to\R$, with $\phi\in L_1(\R^d)\cap C(\R^d)$, and 
such that there 
exists $\tau>d/2$ and $c,c'>0$ for which
\begin{equation}\label{eq:sobolev_embedding} 
c (1 + \norm{2}{\omega}^2)^{-\tau} \leq \hat \phi(\omega) \leq c' (1 + \norm{2}{\omega}^2)^{-\tau},\;\;\omega\in\R^d,
\end{equation}
and domains $\Omega$ that have a Lipschitz boundary. Then there is a norm equivalence $\ns\asymp W_2^\tau(\Omega)$, i.e., $\ns=W_2^\tau(\Omega)$ as sets and the 
norms are equivalent (see e.g. Chapter 10 in \cite{Wendland2005}). More specifically, the upper bound~\eqref{eq:sobolev_embedding} guarantees the 
existence of the embedding $\ns\hookrightarrow W_2^\tau(\Omega)$, while the lower bound guarantees the opposite one.
Although the error bounds of \cite{Wenzel2022c} for $\beta$-greedy account for adaptivity, 
for these kernels they contain an additional term depending on $\log(n)$ raised to some 
power, which seems to be an artifact due to the proof strategy, as demonstrated by numerical experiments and by the fact that the results
obtained for $\beta=0$ are suboptimal with respect to those of \cite{Santin2017x,Wenzel2021a}.

In this paper we make use of a recent result from \cite{Li2023} on general greedy algorithms to formulate a new convergence result 
(Theorem~\ref{thm:new_conv_rates}) that relates the convergence of the $\beta$-greedy algorithms to certain entropy numbers, instead of the Kolmogorov widths 
used before (we refer to Section~\ref{sec:gen_ban_sp} for precise definitions of these notions). This new result applies to general strictly positive definite 
kernels, and it leads to some corollaries when the smoothness of the kernel is taken into account.
These results remove the sub-optimal $\log(n)$ term in the case of Sobolev kernels, proving new rates that are illustrated in
Figure~\ref{fig:rates_of_convergence}.
\begin{figure}
\centering
\includegraphics[width=0.8\textwidth]{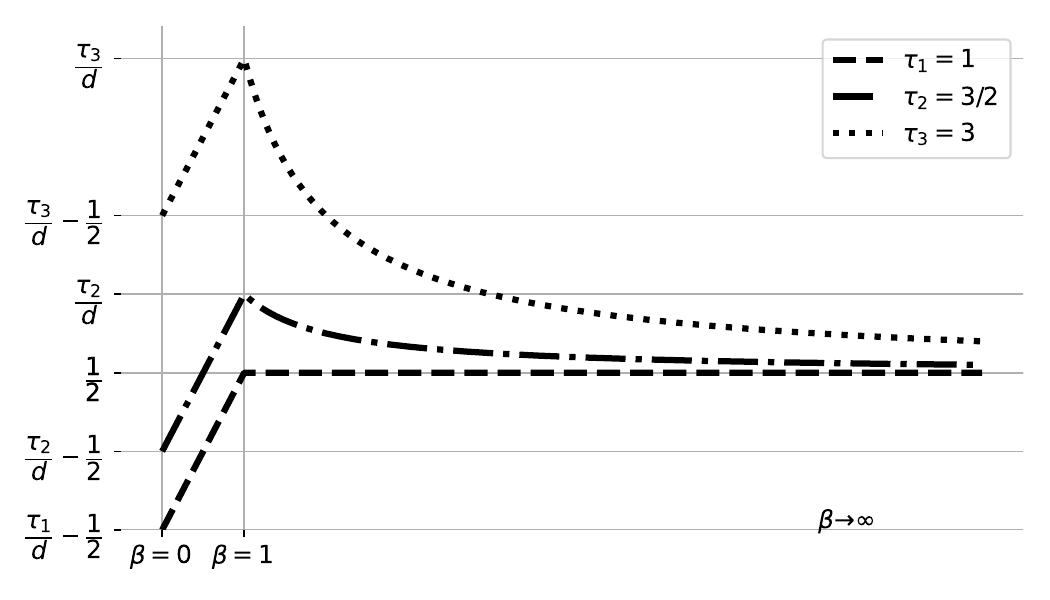}
\caption{Visualization of the new rates of algebraic convergence of the $\beta$-greedy algorithms when $\ns$ is continuously embedded in 
$W^\tau_2(\Omega)$, $\Omega\subset\R^d$, $\tau>d/2$. The lines show the rates for $d=2$ and $\tau_1=1$ (i.e., $\tau_1/d-1/2=0$ -- shown to illustrate the 
limiting case), $\tau_2=3/2$ (i.e., $\tau_2/d-1/2\leq1/2$), and $\tau_3=3$ (i.e., $\tau_3/d-1/2>1/2$). }
\label{fig:rates_of_convergence}
\end{figure}
Apart from this correction, the new estimates are proven to be optimal under different perspectives: we demonstrate with both explicit and 
numerical examples that they cannot be improved if they have to hold for general kernels, and compare them with the rates of uniform best
nonlinear approximation in $\ns$.

The paper is organized as follows. Section~\ref{sec:gen_ban_sp} recalls a fundamental results from \cite{Li2023}, and provides two simple extensions that will 
be used in the following. In Section~\ref{sec:new_conv} we derive the new convergence results, by first introducing a greedy selection rule that works also on 
open bounded sets (Section~\ref{sec:weak_greedy}), and then proving some extensions of the results of ~\cite{Wenzel2022c} (Section~\ref{sec:pre}) to arrive at 
our general convergence result (Section~\ref{sec:convergence_ent}), which is later specified to take into account the smoothness of the kernel 
(Section~\ref{sec:entropy}). These results are discussed in Section~\ref{sec:discussion}, where in particular we address the issue of their optimality, under 
different perspectives. Numerical experiments illustrate some further properties of our new estimates in Section~\ref{sec:numerics}.

\section{A general result on entropy numbers in Hilbert spaces}\label{sec:gen_ban_sp}
We start by introducing a result of \cite{Li2023} that will be the key step of our analysis.

Consider a general Banach space $\calb$, a compact set $K\subset \mathcal \calb$, and let $B(x, \varepsilon)\subset\calb$ be the ball with center $x\in\calb$ 
and
radius $\varepsilon>0$.
We recall (see e.g. \cite{Carl1981,Pinkus1985}) that for $n\in\N$, the $n$-th dyadic entropy number of $K$ in $\calb$ is defined as
\begin{align*}
\varepsilon_n(K, \calb) 
&\coloneqq \inf\left\{\varepsilon > 0: K \text{ can be covered by } 2^n \text{ balls of radius } \varepsilon\right\}\\
&\coloneqq\inf\left\{\varepsilon > 0: \exists x_1, \dots, x_{q}\in \calh, q\leq 2^n: K\subset \bigcup_{i=1}^q B(x_i, \varepsilon)\right\},
\end{align*}
and the Kolmogorov $n$-width of $K$ in $\calb$ is
\begin{equation}\label{eq:kol_width}
d_n(K, \calb)\coloneqq \inf\limits_{V_n\subset \calb, \dim(V_n)\leq n} \sup\limits_{f\in K} \inf\limits_{g\in V_n}\norm{\calb}{f-g}.
\end{equation}
We furthermore recall that the closed and symmetric or absolutely convex hull of a subset $K\subset \calb$ is defined as
\begin{align*}
\conv(K)
= \clos_{\calb}\left\{\sum_{i} c_i g_i: g_i\in K, \sum_i|c_i|\leq 1\right\}.
\end{align*}
With this notation, Lemma 2.1 in \cite{Li2023} states the following, with two minor modifications that we are going to explain in a remark after the statement.
\begin{lemma}\label{lemma:product}
Let $K\subset\calh$ be a precompact set in a Hilbert space $\calh$, let $n\in\N, m\in\N_0$, and let $v_1,\dots, v_{n+m}\in K$ be arbitrary. Define 
$V_0\coloneqq\{0\}$, $V_i\coloneqq\Sp\{v_j, 1\leq j\leq i\}$. 
Then 
\begin{equation*}
\left(\prod_{i=m}^{m+n-1} \norm{\calh}{v_i - \Pi_{V_{i-1}} v_i}\right)^{1/n} 
\leq (n! S_n)^{1/n} \varepsilon_n\left(\conv(K),\calh\right),
\end{equation*}
where $S_n\coloneqq \pi^{n/2} / \Gamma(n/2+1)$ is the volume of the unit ball in $\R^n$.
\end{lemma}
\begin{proof}
Let $w_i\coloneqq v_{i+m-1}$ for $i=1, \dots, n$, and set $W_0\coloneqq\{0\}$, $W_i\coloneqq\Sp\{w_j, 1\leq j\leq i\}=\Sp\{v_j, n\leq j\leq n+i-1\}$ for $1\leq 
i\leq n$. 
Using the fact that $W_{i-m}\subset V_{i-1}$ for $i=m,\dots, m+n-1$, we have
\begin{align}\label{eq:lemma_tmp}
\left(\prod_{i=m}^{n+ m-1} \norm{\calh}{v_i - \Pi_{V_{i-1}} v_i}\right)^{1/n}  
&\leq \left(\prod_{i=m}^{n+m-1} \norm{\calh}{v_i - \Pi_{W_{i-m}} v_i}\right)^{1/n}\\
&= \left(\prod_{i=m}^{n+m-1} \norm{\calh}{w_{i-m+1} - \Pi_{W_{i-m}} w_{i-m+1}}\right)^{1/n}\nonumber\\
&= \left(\prod_{i=1}^{n} \norm{\calh}{w_{i} - \Pi_{W_{i-1}} w_{i}}\right)^{1/n}.\nonumber
\end{align}
Applying now Lemma~2.1 in \cite{Li2023} to the last term in \eqref{eq:lemma_tmp} for the vectors $w_1, \dots, w_n\in K$ gives the result.
\end{proof}

\begin{remark}\label{rem:compactness}
The original result (Lemma 2.1 in \cite{Li2023}) assumes $K$ to be compact. However, the entire statement and its proof only consider $\conv(K)$, and this 
set is compact if $K$ is even just precompact since the base space $\calh$ is complete, see e.g.~\cite{Schaefer1971}\footnote{See the corollary in Chapter 2, 
Section 4.3, in \cite{Schaefer1971}. In the notation of \cite{Schaefer1971}, $\conv(K)$ is called the closed, convex, circled hull of $K$ (see Section 1.3 in 
the same chapter).}.
We thus formulate the result in these terms.
Observe that this assumption is not new in this setting, see e.g.~\cite{Gao2001}.

Moreover, the original lemma is formulated for elements $v_1, \dots, v_n$ instead of $v_{m}, \dots, v_{m+n-1}$ (i.e., for $m=1$). This small modification, of 
which we provide a 
proof for the sake of completeness, will prove crucial in our forthcoming error estimates.
\end{remark}

This result has been used in \cite{Li2023} to refine the convergence rates of \cite{Binev2011} for the Reduced Basis greedy (RB-greedy) algorithm, and its weak 
variant (see Theorem 2.2 in \cite{Li2023}). This algorithm aims at approximating a compact set $K\subset\calh$ representing a certain 
solution manifold of a parametric problem. Roughly speaking, the approximant is built by projecting $K$ into linear subspaces $V_n\coloneqq\Sp\{f_i, 1\leq 
i\leq n\}$, $n\geq 1$, spanned by greedily selected elements
\begin{equation*}
f_n\coloneqq \max_{f\in K} \norm{\calh}{f - \Pi_{V_{n-1}} f}.
\end{equation*}
The theory of \cite{Binev2011} has been the basis for the development of the convergence theory of kernel greedy interpolation algorithms, as discussed in the 
next section.
Before moving to this kernel setting, we briefly comment on the growth of the upper bound in Lemma~\ref{lemma:product}.
\begin{remark}\label{rem:constant}
The term $(n! S_n)^{1/n}$ grows like $\sqrt{n}$, and it is bounded by $\sqrt{2\pi/e}\ n^{1/2}$ for $n\geq n_0$ sufficiently large (see the discussion before 
Corollary 2.3 and after Theorem 4.3 in \cite{Li2023}).
However, since it will be convenient to deal with unrestricted values of $n$, we show here that
\begin{equation}\label{eq:growth_of_rhs}
(n! S_n)^{1/n} \leq \sqrt{5}\ n^{1/2},\ n\in\N.
\end{equation}
To see this, observe that for $n\coloneqq 2m$, $m\geq 1$, $m\in\N$, since $\Gamma(m+1) = m!$ for $m\in\N$ we have
\begin{align*}
(n! S_n)^{1/n}
&= 
\sqrt{\pi} \left(\frac{(2m)!}{m!}\right)^{1/(2m)}
= \sqrt{\pi} \left((m+1) \cdot \ldots \cdot (2m))\right)^{1/(2m)}\\
&\leq \sqrt{\pi} \left(\frac{1}{m}\left((m+1) + \ldots + (2m))\right)\right)^{1/2}
=\sqrt{\pi} \left(\frac{1}{m}\frac{m+1 + 2m}{2}m\right)^{1/2}\\
&=\sqrt{\pi} \left(\frac{3 m+1}{2}\right)^{1/2}
\leq \sqrt{\pi} (2m)^{1/2} = \sqrt{\pi} n^{1/2}.
\end{align*}
Similarly for $n\coloneqq2m -1$, $m\in \N$, $m\geq 1$, we have $\Gamma(n/2+1) = \Gamma(m +1/2)\geq \Gamma(m) = (m-1)!$, and thus
\begin{align*}
(n! S_n)^{1/n}
&\leq \sqrt{\pi} \left(\frac{(2m-1)!}{(m-1)!}\right)^{1/(2m-1)}
= \sqrt{\pi} \left(m \cdot \ldots \cdot (2m-1))\right)^{1/(2m-1)}\\
&\leq \sqrt{\pi} \left(\frac{1}{m}\left(m + \ldots + (2m-1)\right)\right)^{\frac{m}{2m-1}}
= \sqrt{\pi} \left(\frac{1}{m}\frac{m+2m-1}{2}m\right)^{\frac{m}{2m-1}}\\
&= \sqrt{\pi} \left(\frac{3m-1}{2}\right)^{\frac{m}{2m-1}}
% = \sqrt{\pi} \left(\frac{3m-1}{2}\right)^{\frac12 +\left(\frac{m}{2m-1}-\frac12\right)}\\
= \sqrt{\pi} \left(\frac{3m-1}{2}\right)^{\frac12 +\frac{1}{2(2m-1)}}
\leq \sqrt{\pi}(2m-1)^{\frac12 +\frac{1}{2(2m-1)}}\\
&= \sqrt{\pi}(2m-1)^{\frac12}(2m-1)^{\frac{1}{2(2m-1)}}
= \sqrt{\pi}n^{\frac{1}{2}} \sqrt{n^{\frac1{n}}}
\leq \sqrt{\pi e^{1/e}}\ n^{1/2},
\end{align*}
where we used the fact that $\R\ni x\mapsto x^{1/x}$ has a unique maximum at $x=e^{1/e}$.
This proves~\eqref{eq:growth_of_rhs} for all $n\in\N$, since $\max(\sqrt{\pi},  \sqrt{\pi e^{1/e}})\leq \sqrt{5}$.
\end{remark}

\section{A refined convergence theory for greedy kernel interpolation}\label{sec:new_conv}
In this section we obtain the refined convergence rates for the scale of $\beta$-greedy algorithms, but first we formulate in Section~\ref{sec:weak_greedy} a 
generalized form of the $\beta$-greedy selection, which does not require the compactness, but only the boundedness of the input set.

\subsection{Weak greedy algorithms}\label{sec:weak_greedy}
It has been proven in \cite{Santin2017x} that the $P$-greedy algorithm is the instance of the RB-greedy algorithm that is obtained by setting  
$\calh \coloneqq\ns$ for $\Omega$ a compact set, and considering as a compact set $K$ the image of $\Omega$ under the feature map $\Phi:\Omega\to\calh$, 
$\Phi(x)\coloneqq k(\cdot, x)$, i.e.,
\begin{equation*}
K\coloneqq \Phi(\Omega)\coloneqq\left\{k(\cdot, x):x\in\Omega\right\}\subset\ns.
\end{equation*}
We mention that if $\Phi(\Omega)$ is understood as a dictionary in the sense of greedy approximation, then the set $\conv(\Phi(\Omega))$ is 
usually called its variation space.

Observe that $\Phi(\Omega)$ is in fact only precompact if $\Omega$ is just bounded, and it has been proven in \cite{Wenzel2023Thesis} that it is compact if and 
only
if $\Omega$ is closed.  
However, as observed in Remark~\ref{rem:compactness}, its closed and symmetric convex hull
\begin{align*}
\conv(\Phi(\Omega)) = \clos_{\ns}\left\{\sum_{i} c_i k(\cdot, x_i): \sum_i|c_i|\leq 1\right\}
\end{align*}
is a compact set in $\ns$ for any bounded $\Omega$, compact or not. Thus it follows that our theoretical results will apply also to noncompact bounded $\Omega$.

For this to work, however, we need a selection rule that is well defined for an open and bounded set $\Omega$.
To this end, following \cite{Wenzel2023Thesis} we define the \emph{weak $\beta$-greedy selection} as the one that selects,
for an admissible $\gamma\in(0,1]$ and for $\beta\geq 0$, a point
\begin{equation}\label{eq:weak_beta_greedy}
x_{n+1}\in\left\{x\in\Omega :  \eta_\beta(f,X_n)(x)\geq \gamma\sup\limits_{z\in\Omega} \eta_\beta(f, X_n)(z)\right\},
\end{equation}
where $\eta_\beta$ is defined in \eqref{eq:beta_greedy}.
Here, by \emph{admissible} $\gamma$ we mean $\gamma \in (0,1)$ if $\Omega$ is an open set, while $\gamma\in(0,1]$ for a closed $\Omega$.
Observe that indeed this definition allows us in particular to work with an open set $\Omega\subset\R^d$,
where it holds in fact $\eta_\beta(f,X_n)(x_{n+1})\geq \gamma\max_{z\in\bar\Omega} \eta_\beta(f, X_n)(z)$ whenever $k$ is well defined on the closure $\bar 
\Omega$ of $\Omega$ (see \eqref{eq:rkhs_extension}).
Moreover, note that this notion of weak $\beta$-greedy algorithm deviates from the notion of weak greedy used for the RB-greedy algorithm, and more generally in 
approximation theory (see e.g. \cite{DeVore2013}), since here we
cannot relate to the best-approximation error on the right hand side of the inequality, as it does not necessarily exist.

The idea of overcoming the lack of compactness by moving to a weak selection rule, and the corresponding theoretical investigation, has been introduced in 
\cite{Wenzel2022b} for $P$-greedy ($\beta=0$).

\subsection{Preliminaries}\label{sec:pre}
We follow here the argument of \cite{Wenzel2022c} to arrive at Proposition~\ref{prop:thm8}. This is an update of Theorem 8 
in~\cite{Wenzel2022c} that accounts for a modification required for our proof, in order to deal with a weak selection rule.

This change is provided by the following Lemma~\ref{lemma:weak_beta_greedy}, which is a very simple generalization of Lemma 7 in \cite{Wenzel2022c}.
Since the result follows by simply tracking the term $\gamma$ in the original argument, we postpone its proof to Appendix~\ref{sec:appendix}. 
From here on, we use sometimes the shorthand notation $r_i\coloneqq f - I_{X_i}f$, provided $f\in \ns$, $X_i\subset \Omega$, and $i\in\N_0$ are clear from the 
context.
\begin{lemma}\label{lemma:weak_beta_greedy}
Let $\Omega\subset\R^d$ be bounded and assume that $k$ is well defined on $\bar\Omega$.
Then any weak $\beta$-greedy algorithm with $\beta \in [0,\infty]$ and admissible $\gamma\in(0,1]$, applied to $f \in \ns$, satisfies the following.
\begin{enumerate}
\item If $\beta \in [0, 1]$ then
\begin{equation}\label{eq:bound_ri_first}
\norm{L_\infty(\Omega)}{r_i} 
\leq \gamma^{-1}|r_i(x_{i+1})|^{\beta}  P_{X_i}(x_{i+1})^{1-\beta}  \norm{\calh_k(\bar \Omega)}{r_i}^{1-\beta},\; i\in\N_0.
\end{equation}
\item If $\beta \in (1, \infty]$ (with $1/\infty \coloneqq 0$) then
\begin{equation}\label{eq:bound_ri_second}
\norm{L_\infty(\Omega)}{r_i} 
\leq \gamma^{-1}\frac{|r_i(x_{i+1})|}{P_{X_i}(x_{i+1})^{1-1/\beta}}  \norm{L_\infty(\Omega)}{P_{X_i}}^{1 - 1/\beta},\; i\in\N_0.
\end{equation}
\end{enumerate}
\end{lemma}

\begin{remark}
In the proof of this lemma (see Appendix \ref{sec:appendix}) we actually prove that \eqref{eq:bound_ri_second} holds even with a
term $\gamma^{-1/\beta}$ in place of $\gamma^{-1}$. However, this is rather due to a scaling issue in the definition of the selection rule 
~\eqref{eq:weak_beta_greedy}, which does not normalize the value of $\eta_\beta$ for different $\beta$.
To avoid dealing with this issue and for coherence with~\cite{Wenzel2022c}, we keep here the simpler term $\gamma^{-1}$. 
We refer however to~\cite{Wenzel2023Thesis} for a different parametrization of the scale of $\beta$-greedy algorithms for $\beta>1$, which 
leads to a different scaling of this factor.
\end{remark}

With this lemma we can get the required update of Theorem 8 in \cite{Wenzel2022c}.
\begin{proposition}\label{prop:thm8}
Let $\Omega\subset\R^d$ be bounded, and assume that $k$ is well defined on $\bar\Omega$ and \eqref{eq:rkhs_embedding} holds.
Then any weak $\beta$-greedy algorithm with $\beta \in [0,\infty]$ and admissible $\gamma\in(0,1]$, applied to a function $f \in \ns$, satisfies
\begin{align} \label{eq:final_result_1}
\left(\prod_{i=n+1}^{2n} \norm{L_\infty(\Omega)}{r_i} \right)^\frac1n
\leq \gamma^{-1} n^{-\frac{\min(\beta,1)}{2}} \norm{\ns}{r_{n+1}}\left(\prod_{i=n+1}^{2n} P_{X_i}(x_{i+1})^{\frac1{\max(\beta,
1)}} \right)^\frac1n,\;\; n\in\N.
\end{align}
\end{proposition}
\begin{proof}
We use Lemma~6 in \cite{Wenzel2022c}, which states that
\begin{equation}\label{eq:lemma6_in_wenzel2022c}
\left(\prod_{i=n+1}^{2n} \left|r_i(x_{i+1})\right| \right)^{1/n}
\leq n^{-1/2} \norm{\ns}{r_{n+1}} \left(\prod_{i=n+1}^{2n} P_{X_i}(x_{i+1})\right)^{1/n}.
\end{equation}
We split the proof in two parts.
\begin{enumerate}
\item 
For $\beta \leq 1$ we use equation \eqref{eq:bound_ri_first} in Lemma~\ref{lemma:weak_beta_greedy}, and \eqref{eq:rkhs_embedding} to write
\begin{align*}
\norm{{L_\infty(\Omega)}}{r_i}
&\leq \gamma^{-1} |r_i(x_{i+1})|^{\beta} P_{X_i}(x_{i+1})^{1-\beta} \norm{\calh_k(\bar\Omega)}{r_i}^{1-\beta}\\
&\leq  \gamma^{-1} |r_i(x_{i+1})|^{\beta} P_{X_i}(x_{i+1})^{1-\beta} \norm{\ns}{r_i}^{1-\beta}, \;1\leq i \leq n.
\end{align*}
Taking the geometric mean of both sides for $i=n+1, \dots, 2n$, and using~\eqref{eq:lemma6_in_wenzel2022c}, we get
\begin{align*}
\left(\prod_{i=n+1}^{2n} \norm{L_\infty(\Omega)}{r_i} \right)^{\frac1n}
&\leq  \gamma^{-1} n^{-\beta/2} \norm{\ns}{r_{n+1}}^\beta \left(\prod_{i=n+1}^{2n} P_{X_i}(x_{i+1}) \right)^\frac{\beta}{n}
\left(\prod_{i=n+1}^{2n} P_{X_i}(x_{i+1})^{1-\beta} \right)^\frac{1}{n}\\
&\quad \cdot\left(\prod_{i=n+1}^{2n} \norm{\ns}{r_i}^{1-\beta}\right)^{\frac1n}\\
&\leq \gamma^{-1} n^{-\beta/2} \norm{\ns}{r_{n+1}} \left(\prod_{i=n+1}^{2n} P_{X_i}(x_{i+1}) \right)^\frac{1}{n},
\end{align*}
where we used the fact that $\norm{\ns}{r_i}\leq\norm{\ns}{r_{n+1}}$ for all $n+1\leq i\leq 2n$. This proves the statement for $\beta \in[0,1]$.

\item For $\beta > 1$ we proceed similarly using equation \eqref{eq:bound_ri_second} in Lemma~\ref{lemma:weak_beta_greedy}, to get
\begin{equation*}
\norm{L_\infty(\Omega)}{r_i} 
\leq \gamma^{-1}\frac{|r_i(x_{i+1})|}{P_{X_i}(x_{i+1})^{1-1/\beta}} \norm{L_\infty(\Omega)}{P_{X_i}}^{1 - 1/\beta}
\leq \gamma^{-1}\frac{|r_i(x_{i+1})|}{P_{X_i}(x_{i+1})^{1-1/\beta}},
\end{equation*}
where $\norm{L_\infty(\Omega)}{P_{X_i}} \leq 1$ by \eqref{eq:kernel_normalization}. Taking again the same geometric mean, and using 
again~\eqref{eq:lemma6_in_wenzel2022c}, we obtain
\begin{align*}
\left(\prod_{i=n+1}^{2n} \norm{L_\infty(\Omega)}{r_i} \right)^{1/n} 
&\leq \gamma^{-1} n^{-1/2} \norm{\ns}{r_{n+1}} \left(\prod_{i=n+1}^{2n} P_{X_i}(x_{i+1}) \right)^{1/n} \left(\prod_{i=n+1}^{2n} 
\frac{1}{P_{X_i}(x_{i+1})^{1-1/\beta}} 
\right)^{1/n}\\
&= \gamma^{-1} n^{-1/2} \norm{\ns}{r_{n+1}} \left(\prod_{i=n+1}^{2n} P_{X_i}(x_{i+1})^{1/\beta}\right)^{1/n},
\end{align*}
which gives the desired result for $\beta\in(1,\infty]$.
\end{enumerate}
\end{proof}

\subsection{Convergence in terms of entropy numbers}\label{sec:convergence_ent}
We can now combine the results of the previous section with Lemma~\ref{lemma:product} and obtain the new convergence rates.
\begin{theorem}\label{thm:new_conv_rates}
Let $\Omega\subset\R^d$ be bounded, and assume that $k$ is well defined on $\bar\Omega$ and \eqref{eq:rkhs_embedding} holds.
Let $\varepsilon_{n}\coloneqq \varepsilon_n\left(\conv(\Phi(\Omega)),\ns\right)$. 

Then any weak $\beta$-greedy algorithm with $\beta \in [0,\infty]$ and admissible $\gamma\in(0,1]$, applied to a function $f \in \ns$, satisfies
\begin{equation*}
\min\limits_{n+1\leq i \leq 2n} \norm{L_\infty(\Omega)}{f - I_{X_i} f}
\leq \gamma^{-1}\left(\sqrt{5} n^{\frac{1-\beta}{2}} \varepsilon_{n}\right)^{\frac{1}{\max(1,\beta)}}\norm{\ns}{f - I_{X_{n+1}} f},
\;\;n\in\N.
\end{equation*}
\end{theorem}
\begin{proof}
The proof follows the argument of Corollary 11 in \cite{Wenzel2022c}, with the modification that we replace Theorem 8 in \cite{Wenzel2022c} with 
Proposition~\ref{prop:thm8}, and Corollary 2 in \cite{Wenzel2022c} with Lemma~\ref{lemma:product}.

First we observe that
\begin{equation*}
\min\limits_{n+1\leq i \leq 2n} \norm{L_\infty(\Omega)}{r_i} 
\leq \left(\prod_{i=n+1}^{2n} \norm{L_\infty(\Omega)}{r_i} \right)^{1/n},
\end{equation*}
and thus Proposition~\ref{prop:thm8} gives
\begin{align*}
\min\limits_{n+1\leq i \leq 2 n} \norm{L_\infty(\Omega)}{r_i}
\leq  \gamma^{-1} n^{-\min(1, \beta)/2}  \norm{\ns}{r_{n+1}} \left(\prod_{i=n+1}^{2n} P_{X_i}(x_{i+1}) \right)^{1/(n\max(1, \beta))}.
\end{align*}
We now set $K\coloneqq \Phi(\Omega)$ and $v_i\coloneqq k(\cdot, x_i)\in K$, with $x_1, \dots, x_{2n}$ selected by the greedy algorithm, so that 
$V_i\coloneqq\Sp\{v_j, 
1\leq j\leq i\}=V(X_i)$. It follows from \eqref{eq:power_as_proj} that
\begin{equation*}
\norm{\ns}{v_i - \Pi_{V_{i-1}} v_i} 
=
\norm{\ns}{k(\cdot, x_i) - \Pi_{V(X_{i-1})} k(\cdot, x_i)}
= P_{X_{i-1}}(x_i),
\end{equation*}
and thus Lemma~\ref{lemma:product} with $m = n$, and Remark~\ref{rem:constant}, give
\begin{equation*}
\left(\prod_{i=n+1}^{2n} P_{X_{i}}(x_{i+1}) \right)^{1/n}
=\left(\prod_{i=n}^{2n-1} P_{X_{i-1}}(x_{i}) \right)^{1/n}
\leq \sqrt{5} n^{1/2} \varepsilon_n.
\end{equation*}
Combining these bounds leads to the estimate
\begin{align*}
\min\limits_{n+1\leq i \leq 2n} \norm{L_\infty(\Omega)}{r_i}
&\leq \gamma^{-1} n^{-\min(1,\beta)/2} \norm{\ns}{r_{n+1}} \left(\sqrt{5} n^{1/2} \varepsilon_n\right)^{1/\max(1,\beta)}\\
&= \gamma^{-1}\norm{\ns}{r_{n+1}} \left(\sqrt{5} n^{-\min(1,\beta)\max(1,\beta)/2} n^{1/2}
\varepsilon_n\right)^{1/\max(1,\beta)}\nonumber\\
&= \gamma^{-1} \norm{\ns}{r_{n+1}} \left(\sqrt{5} n^{(1-\beta)/2} \varepsilon_n\right)^{1/\max(1,\beta)},\nonumber
\end{align*}
as in the statement.
\end{proof}

The result of Theorem~\ref{thm:new_conv_rates} is applicable to a broad class of kernels, and it is of interest in itself as it covers also the case of a weak 
algorithm on a non-compact set $\Omega$, and because it provides an explicit relation of the greedy interpolation error with the decay of an entropy number. 
We combine it with estimates on the decay of $\varepsilon_n$ for specific kernels in the next section.

\subsection{Entropy numbers of the kernel dictionary and algebraic error rates}\label{sec:entropy}
It remains to provide estimates for $\varepsilon_n(\conv(\Phi(\Omega), \ns)$, for which we use the theory of \emph{smoothly parametrized dictionaries} in the 
sense 
of \cite{Siegel2022b}. We recall here the fundamental definitions and use them to prove a decay of these entropy numbers.

For a function $f:\Omega \to\R$, $m\in\N_0$, $\alpha\in(0,1]$, $s\coloneqq m+\alpha$, we write $f\in \Lip(s, L_\infty(\Omega))$ if its Lipschitz seminorm
\begin{equation*}
\seminorm{\Lip(s,L_\infty(\Omega))}{f}\coloneqq \max\limits_{|\zeta|=m}\sup_{y\neq z, y,z\in\Omega} \frac{\left|D^\zeta f(y) - D^\zeta
f(z)\right|}{|y-z|^\alpha}
\end{equation*}
is bounded (see e.g. \cite{Lorentz1996a}). This notion extends to $\ns$-valued maps $\mathcal F:\Omega\to\ns$ (Definition 2 in
\cite{Siegel2022b}), where we say that $\mathcal F$ is of smoothness class $\Lip_\infty(s, \ns)$ if there exists $C>0$ such that, for every $u\in 
\ns$, the function 
$f_u:\Omega\to\R$ defined by
\begin{equation*}
f_u(x)\coloneqq\inner{\ns}{u, \mathcal F(x)},\; x\in\Omega,
\end{equation*}
satisfies
\begin{equation*}
\seminorm{\Lip(s, L_\infty(\Omega))}{f_u} \leq C \norm{\ns}{u}.
\end{equation*}
Dictionaries that are contained in the image of such an $\mathcal F$ are named \emph{smoothly parametrized} (Section 3 in \cite{Siegel2022b}), and rates of 
decay of the entropy number of their variation space is provided in Theorem 4 of \cite{Siegel2022b}. Collecting all these facts, we prove the following result 
for our case of interest, where we recall from Section~\ref{sec:weak_greedy} that $\Phi:\Omega\to\ns$, $\Phi(x)\coloneqq k(\cdot, x)$, $x\in\Omega$.

\begin{proposition}\label{prop:entropy_as_lip}
Let $\Omega\subset \R^d$ and let $s>0$. 
Then $\Phi\in\Lip_\infty(s, \ns)$ if and only if $\ns$ is 
continuously embedded\footnote{By this we mean that the space is contained into $\Lip(s, L_\infty(\Omega))$, and the inequality~\eqref{eq:lip_embedding} holds. 
This is a slight
abuse of the notion, since $\Lip$ is not a normed space.} into 
$\Lip(s, L_\infty(\Omega))$, i.e., there exists an embedding constant $c_e>0$ such that
\begin{equation}\label{eq:lip_embedding}
\seminorm{\Lip(s, L_\infty(\Omega))}{u} \leq c_e \norm{\ns}{u}\;\;\fa u\in \ns.
\end{equation}
In this case, if the boundary of $\Omega$ is sufficiently smooth (say, $C^\infty$), there are constants $C, C'>0$ such that
\begin{equation}\label{eq:entropy_decay}
\varepsilon_n(\conv(\Phi(\Omega)),\ns) \leq C n^{-\frac s d-\frac12},\;\; n\in\N,
\end{equation}
and
\begin{equation}\label{eq:kolmog_decay}
d_n(\conv(\Phi(\Omega)),\ns) \leq C' n^{-\frac s d},\;\; n\in\N.
\end{equation}
\end{proposition}
\begin{proof}
This is a straightforward calculation, since from the definition of $\Phi$, for all $u\in\ns$ it holds
\begin{equation*}
f_u(x)\coloneqq\inner{\ns}{u, \Phi(x)} = \inner{\ns}{u, k(\cdot, x)} = u(x),\; x\in\Omega,
\end{equation*}
and thus clearly 
$\seminorm{\Lip(s,L_\infty(\Omega))}{f_u}
\leq c_e \norm{\ns}{u}
$
if and only if 
$
\seminorm{\Lip(s,L_\infty(\Omega))}{u}
\leq c_e \norm{\ns}{u}
$.

If now $\Phi\in\Lip_\infty(s, \ns)$, Theorem 4 in \cite{Siegel2022b} provides the bound \eqref{eq:entropy_decay}, and Theorem 5 in \cite{Siegel2022b} 
provides the bound \eqref{eq:kolmog_decay}.
\end{proof}
\begin{remark}
The smoothness requirement on the boundary of $\Omega$ are not exactly specified in \cite{Siegel2022b}, even if they are related to the polynomial 
approximability of $\Phi$, and they affect in this respect the constants $C, C'$. We are not interested here in analyzing this issue more precisely than this, 
so we simply assume that $\Omega$ has a $C^\infty$ boundary and refer to the proof of Theorem 4 in \cite{Siegel2022b} for further details.
\end{remark}

\begin{remark}
Other approaches (see \cite{Gao2001}) allow one to obtain estimates on the decay of $\varepsilon_n(\conv(\Phi(\Omega)),\ns)$ from the decay of 
$\varepsilon_n(\Phi(\Omega),\ns)$, which needs however to be in turn estimated in some way. The method presented here seems more suitable for our goals, as it 
simply relates the decay to the smoothness of $k$.
\end{remark}

Combining Proposition~\ref{prop:entropy_as_lip} and Theorem~\ref{thm:new_conv_rates} gives our main result.
\begin{corollary}\label{thm:new_conv_rates_explicit}
Under the hypotheses of Theorem~\ref{thm:new_conv_rates} and Proposition~\ref{prop:entropy_as_lip}, let $s>0$ and assume that $\ns$ is continuously embedded 
into $\Lip(s,
L_\infty(\Omega))$.

Then any weak $\beta$-greedy algorithm with $\beta \in [0,\infty]$ and admissible $\gamma\in(0,1]$, applied to a function $f \in
\ns$, satisfies
\begin{equation*}
\min\limits_{n+1\leq i \leq 2 n} \norm{L_\infty(\Omega)}{f - I_{X_i} f}
\leq c\ \gamma^{-1}\left(n^{-\frac{s}{d}-\frac{\beta}{2}}\right)^{\frac{1}{\max(1,\beta)}}\norm{\ns}{f-I_{X_{n+1}}f}, \;\;n\in\N,
\end{equation*}
where $c\coloneqq \left(C \sqrt{5}\right)^{\frac{1}{\max(1,\beta)}}$ and $C$ is the constant of \eqref{eq:entropy_decay}.
\end{corollary}

The condition that $\ns$ is continuously embedded into $\Lip(s, L_\infty(\Omega))$ has notable examples. 
Namely, it clearly holds $C^{m,\alpha}(\Omega)\hookrightarrow \Lip(s, L_\infty(\Omega))$ for $m=\lfloor s\rfloor$ and $\alpha = s-m$, since the space has a norm
\begin{equation*}
\norm{C^{m,\alpha}(\Omega)}{u} = \norm{C^{m}(\Omega)}{u} + \seminorm{\Lip(s, L_\infty(\Omega))}{u}, \;\;u\in C^{m,\alpha}(\Omega).
\end{equation*}
The embedding $\ns\hookrightarrow C^{m,\alpha}(\Omega)$ exists if $k\in C^{2m, 2\alpha}(\Omega)$, as can be seen with an explicit computation and using the 
reproducing property, and thus in this case $\ns\hookrightarrow \Lip(s,L_\infty(\Omega))$.
Moreover, if $\ns$ is continuously embedded into $W_2^\tau(\Omega)$ (i.e., the right inequality in~\eqref{eq:sobolev_embedding} holds) and $\Omega$
has a sufficiently smooth boundary, then the Sobolev embedding theorem guarantees that $W_2^\tau(\Omega)\hookrightarrow C^{m,\alpha}(\Omega)$ if $\tau>d/2$ and 
$s\coloneqq 
m+\alpha = \tau - d/2$ (i.e., $m\coloneqq \lfloor\tau-d/2\rfloor$, $\alpha\coloneqq \tau - d/2 - m$). We state this fact as the following corollary.
\begin{corollary}\label{thm:new_conv_rates_explicit_sobolev}
Let $\Omega\subset \R^d$ have a Lipschitz boundary, let $\tau>d/2$, and assume that $\ns$ is continuously embedded into $W_2^\tau(\Omega)$.
Then, under the assumptions of Corollary~\ref{thm:new_conv_rates_explicit} and Proposition~\ref{prop:entropy_as_lip}, for all $f\in \ns$ it holds
\begin{equation*}
\min\limits_{n+1\leq i \leq 2 n} \norm{L_\infty(\Omega)}{f - I_{X_i} f}
\leq c\ \gamma^{-1}\left(n^{-\frac{\tau}{d}+\frac{1-\beta}{2}}\right)^{\frac{1}{\max(1,\beta)}}\norm{\ns}{f-I_{X_{n+1}}f}, \;\;n\in\N,
\end{equation*}
where $c>0$ is as in Corollary~\ref{thm:new_conv_rates_explicit}.
\end{corollary}

This corollary removes the $\log(n)$ term that was present in Corollary 11 and Corollary 12 in \cite{Wenzel2022c}, and that was conjectured to be an 
artifact due to the proof strategy, rather than an actual necessity (see the discussion in Section 5.2 in \cite{Wenzel2022c}). In more details, Corollary 12 in 
\cite{Wenzel2022c} states that, under the same assumptions as 
Corollary~\ref{thm:new_conv_rates_explicit_sobolev}, for $\Omega$ compact and $\gamma=1$, there is $c'>0$ such that
\begin{equation}\label{eq:bound_with_log}
\min_{n+1\leq i\leq 2n}\norm{L_\infty(\Omega)}{f - I_{X_i}f}
\leq c' \left(n^{-\frac{\tau}{d}-\frac{1-\beta}{2}}\right)^{\frac{1}{\max\{1, \beta\}}} \log(n)^{\frac{\tau/d-1/2}{\max\{1, \beta\}}} 
\norm{\ns}{f-I_{X_{n+1}}f}.
\end{equation}
The logarithmic term is now removed by Corollary~\ref{thm:new_conv_rates_explicit_sobolev}.

We discuss some optimality aspects of this result in Section~\ref{sec:discussion}, and conclude this section with three remarks.
\begin{remark} 
The bounds of Corollary~\ref{thm:new_conv_rates_explicit} and Corollary~\ref{thm:new_conv_rates_explicit_sobolev} can be formulated over indices $1,\dots, m$ 
for all $m\in\N$, up to changing the constant in the bound. 
Indeed, for all $m\in\N$ we have $2 \lfloor m/2\rfloor\leq m$, so 
\begin{equation}\label{eq:reindex}
\min\limits_{1\leq i \leq m} \norm{L_\infty(\Omega)}{r_i}
\leq \min\limits_{\lfloor m/2\rfloor\leq i \leq 2 \lfloor m/2\rfloor} \norm{L_\infty(\Omega)}{r_i},\; m\in\N.
\end{equation}
Moreover, setting $n\coloneqq \lfloor m/2\rfloor$ we have $n\geq m/4$ for $m\in\N$, so that Corollary~\ref{thm:new_conv_rates_explicit} gives
\begin{equation*}
\min\limits_{n+1\leq i \leq 2 \lfloor m/2\rfloor} \norm{L_\infty(\Omega)}{f - I_{X_i} f}
\leq c\ 
\gamma^{-1}\left(4^{\frac{s}{d}+\frac{\beta}{2}}\right)^{\frac{1}{\max(1,\beta)}}\left(m^{-\frac{s}{d}-\frac{\beta}{2}}\right)^{\frac{1}{\max(1,\beta)}
}\norm{\ns}{f-I_{X_{\lfloor m/2\rfloor+1}}f},
\end{equation*}
and in particular using~\eqref{eq:reindex} we have
\begin{equation}\label{eq:bound_without_two}
\min\limits_{1\leq i \leq m} \norm{L_\infty(\Omega)}{f - I_{X_i} f}
\leq c''\ 
\gamma^{-1}\left(m^{-\frac{s}{d}-\frac{\beta}{2}}\right)^{\frac{1}{\max(1,\beta)} }\norm{\ns}{f},
\end{equation}
which could be easier to handle in some cases.
The same applies to Corollary~\ref{thm:new_conv_rates_explicit_sobolev}.
\end{remark}
\begin{remark}
The error estimate of Corollary~\ref{thm:new_conv_rates_explicit} and Corollary~\ref{thm:new_conv_rates_explicit_sobolev} differ from those of 
\cite{Santin2017x,Wenzel2021a,Wenzel2022c} in another important aspect.
Namely, in the papers \cite{Santin2017x,Wenzel2021a,Wenzel2022c} the existence of a worst case error bound like
\begin{equation}\label{eq:rem_tmp}
\norm{L_\infty(\Omega)}{f - I_{X_n} f}\leq C n^{-\delta} \norm{\ns}{f}\;\;\fa f\in \ns,
\end{equation}
for some $\delta>0$, $C>0$, is used to deduce a convergence rate of $\beta$-greedy.
In this paper, instead, the convergence rate is based on a smoothness assumption on the kernel.

This difference leaves an interesting question open. Indeed, the same smoothness assumptions that we are requiring on the kernel are known to prove a worst case 
error~\eqref{eq:rem_tmp} for suitable $\delta, C>0$ (see e.g. Theorem 11.11 in \cite{Wendland2005} for the case $\ns\hookrightarrow C^{m,
\alpha}(\Omega)$).
What is unknown is if the existence of a worst case error bound like~\eqref{eq:rem_tmp} implies a decay $\varepsilon_n(\Phi(\Omega), \ns)\leq c 
n^{-\delta-1/2}$. This would imply, following the argumentation of the present paper, that a worst case error estimate proves the convergence of the 
$\beta$-greedy algorithm with the correct rate.
\end{remark}

\begin{remark}
In the case of native spaces where the worst-case approximation has an exponential decay rate, i.e.,
there exists a sequence $\{X_n\}_{n\in\N}\subset \Omega$ such that for some $\delta>0$ it holds
\begin{equation*}
\norm{L_\infty(\Omega)}{f - \Pi_{X_n} f} \leq C e^{-c n^\delta} \norm{\ns}{f}\;\;\fa f \in \ns,
\end{equation*}
Corollary~\ref{thm:new_conv_rates_explicit} gives no new information over existing estimates. Indeed, in this case there was no spurious 
logarithmic term in the original result (Corollary 11 in \cite{Wenzel2022c}).
However, for kernels embedded in Sobolev spaces of arbitrary order (i.e., those whose Fourier transform in \eqref{eq:sobolev_embedding} decays faster than any 
polynomial), it could be possible to apply Proposition~\ref{prop:entropy_as_lip} with an arbitrary large $s>0$, up to estimating the growth of the constants as 
a 
function of $s$, as it is typically done for these kernels (see e.g. Section 11.4 in \cite{Wendland2005}, or the appendix in \cite{Rieger2008b}). 
\end{remark}

\subsection{Weakened convergence for weaker selections}
In certain situations it is of interest to consider a weaker selection rule, where one allows the quality of the selected points to degrade with the iteration, 
instead of being bounded to be a factor $\gamma$ away from the optimal selection (if this exists).
In this respect, we notice that all the estimates that lead to Theorem~\ref{thm:new_conv_rates} and its corollaries are valid for any finite $n\in\N$, and not 
only in the limit as $n\to\infty$. Indeed, unlike the
theory of 
\cite{Binev2011}, all the result used here do not need to assume the convergence of the various sequences involved in the estimates, and the result for 
quantities indexed by $n$ follow solely from estimates of (possible other) quantities with index $1\leq i\leq n$. 
In particular, it is possible to consider a sequence $\{\gamma_n\}_{n\in\N}\subset(0,1]$ of parameters, and define
\begin{equation}\label{eq:weak_beta_greedy_sequence}
x_{n+1}\in\left\{x\in\Omega :  \eta_\beta(f,X_n)(x)\geq \gamma_n\sup\limits_{z\in\Omega} \eta_\beta(f, X_n)(z)\right\}.
\end{equation}
This leads to the following modification of Theorem~\ref{thm:new_conv_rates}.
\begin{corollary}
Let $\Omega\subset\R^d$ be bounded, and assume that $k$ is well defined on $\bar\Omega$ and \eqref{eq:rkhs_embedding} holds.
Let $\varepsilon_{n}\coloneqq \varepsilon_n\left(\conv(\Phi(\Omega)),\ns\right)$. 

Any weak $\beta$-greedy algorithm which selects points according to \eqref{eq:weak_beta_greedy_sequence} with $\beta \in [0,\infty]$ and 
$\{\gamma_n\}_{n\in\N}\subset(0,1]$, applied to a function $f \in \ns$, satisfies  
\begin{equation*}
\min\limits_{n+1\leq i \leq 2n} \norm{L_\infty(\Omega)}{f - I_{X_i} f}
\leq \bar\gamma^{-1}\left(\sqrt{5} n^{\frac{1-\beta}{2}} \varepsilon_{n}\right)^{\frac{1}{\max(1,\beta)}}\norm{\ns}{f-I_{X_{n+1}}f},
\;\;n\in\N,
\end{equation*}
where $\bar\gamma\coloneqq \min_{n+1\leq i\leq 2n}\gamma_i$ or $\bar\gamma\coloneqq \left(\prod_{i=n+1}^{2n}\gamma_i\right)^{1/n}$.
\end{corollary}
The statement follows by tracking the $\gamma_i$ quantities within Proposition \ref{prop:thm8} more closely.

This implies for example that one could use~\eqref{eq:weak_beta_greedy_sequence} with $\gamma_i>c' i^{-p}$ for some $c', p>0$ (as done e.g. 
in~\cite{Santin2023x}), and
the bound of Corollary~\ref{thm:new_conv_rates_explicit_sobolev} would become
\begin{equation*}
\min\limits_{n+1\leq i \leq 2 n} \norm{L_\infty(\Omega)}{f - I_{X_i} f}
\leq c\cdot c'\cdot n^{-\frac{2\tau + d(1-\beta)}{2d\max(1,\beta)}+p}\norm{\ns}{f-I_{X_{n+1}}f}, \;\;n\in\N.
\end{equation*}
This means that the convergence rate is slower by an algebraic rate of $p$ compared to the case of a constant $\gamma$.

\section{Discussion of the results and of their optimality}\label{sec:discussion}
We discuss here the results obtained in the previous section. In particular, we first prove that they are optimal in certain respects 
(Section~\ref{sec:opt}), and analyze their relation with optimal nonlinear approximation (Section~\ref{sec:non_lin}). Both these sections leave some questions 
open, and these are partially addressed by an explicit example in Section~\ref{sec:bb_example}.
We remark that in this section the constants $c,c'>0$ are different from one estimate to the other.

In this section we will make use of known lower bounds that can be obtained for kernels whose native space is norm equivalent to $W_2^\tau(\Omega)$, i.e., those 
which satisfy the two inequalities~\eqref{eq:sobolev_embedding}.
In this case, it can be proven that the power function cannot decay too fast. In more detail (see e.g. Theorem
12 in~\cite{Wenzel2021a}) there is $c>0$ such that for any arbitrary point distribution it holds
\begin{equation}\label{eq:lower_bound_pf}
c n^{-\frac{\tau}{d}+\frac{1}{2}} \leq \norm{L_\infty(\Omega)}{P_{X_{n}}},\;\;n\in\N.
\end{equation}
This rate is matched by sequences $\{X_n\}_{n\in\N}\subset\Omega$ of quasi-uniform points, i.e., setting
\begin{equation*}
h_X\coloneqq \sup\limits_{z\in\Omega}\min\limits_{x\in X} \|x - z\|_2,\;\; 
q_X\coloneqq \min\limits_{x\neq z\in X} \|x-z\|_2,
\end{equation*}
there exists $\rho\in\R$ such that $\rho_{X_n}\coloneqq h_{X_n}/q_{X_n}<\rho$ for all $n\in\N$. 
For such sequences, it is additionally known (see 
\cite{Schaback1995}) that there is $c'>0$ such that the smallest eigenvalue of the kernel matrix $A_{k, X_n}\coloneqq \left(k(x_i,x_j)\right)_{1\leq i,j\leq 
n}\in\R^{n\times 
n}$ satisfies
\begin{equation}\label{eq:lower_bound_eig}
c' n^{-2\tau/d+1}\leq \lambda_{\min}(A_{k, X_n}),\;\;n\in\N.
\end{equation}
We will use both lower bounds in the following.

\subsection{Optimality of some components of the estimates}\label{sec:opt}
We analyze here the main building blocks of the estimates, and show that some of them cannot be improved if they have to hold under sufficiently general 
assumptions. 

We start with Lemma~\ref{lemma:product}, and show that it has certain consequences on the decay of the eigenvalues of the kernel matrix, which is lower
bounded. Observe that other aspects of the optimality of Lemma~\ref{lemma:product} have already been discussed in~\cite{Li2023}.
\begin{proposition}
Let $A_{k, X_n}\coloneqq \left(k(x_i,x_j)\right)_{1\leq i,j\leq n}\in\R^{n\times 
n}$ be the kernel matrix associated with a set of points $X_n\subset\Omega$. Then the smallest eigenvalue and the determinant of $A_{k, X_n}$ satisfy
\begin{equation}\label{eq:unc_princ}
\lambda_{\min}(A_{k, X_n}) \leq \det(A_{k, X_n})^{\frac1n} \leq 5\ n\varepsilon_n^2\left(\conv(\Phi(\Omega)),\ns\right).
\end{equation}
Under the assumptions of Proposition~\ref{prop:entropy_as_lip}, if additionally $\ns$ is norm equivalent to $W_2^\tau(\Omega)$,
and if $\{X_n\}_{n\in\N}\subset\Omega$ is a sequence of quasi-uniform points, then there are constants $c, c'>0$ such that 
\begin{equation}\label{eq:asympt_eigen_sobolev}
c' n^{-2\tau/d+1}\leq \lambda_{\min}(A_{k, X_n})\leq \det(A_{k, X_n})^{\frac1n} \leq c n^{-2\tau/d+1}.
\end{equation}
Especially, the rate of Lemma~\ref{lemma:product} cannot be improved.
\end{proposition}
\begin{proof}
Let $A_{k, X_n}= L L^T$ be a Cholesky decomposition, which exists since the matrix is positive definite.
We define $v_i\coloneqq k(\cdot, x_i)\in\ns$ for an arbitrary enumeration of $X_n$, and set $V_i\coloneqq \Sp\{v_j, 1\leq j\leq i\}$ as in 
Lemma~\ref{lemma:product}. Then $u_i\coloneqq v_i - \Pi_{V_{i-1}} v_i$ are the vectors obtained from the Gram-Schmidt orthogonalization (i.e., no
normalization) of the vectors $v_i$. These are the non-normalized Newton basis of $V(X_n)$ (see~\cite{Muller2009,Pazouki2011}), whose norms satisfy
$\norm{\ns}{u_i}=L_{ii}$, i.e., they are the diagonal elements of the Cholesky factor (Remark 8 in \cite{Muller2009}). It follows that
\begin{equation*}
\prod_{i=1}^n \norm{\ns}{v_i-\Pi_{V_{i-1}} v_i} = \prod_{i=1}^n L_{ii} = \det(L) = \sqrt{\det(A_{k, X_n})},
\end{equation*}
and thus applying Lemma~\ref{lemma:product} with $m=1$ gives \eqref{eq:unc_princ}.

We can now use the same argument as in  Corollary~\ref{thm:new_conv_rates_explicit_sobolev} to upper bound $\varepsilon_n\left(\conv(\Phi(\Omega)),\ns\right)$ 
for kernels whose native space is embedded in $W_2^\tau(\Omega)$, and inserting this estimate into~\eqref{eq:unc_princ} gives the upper bound in 
\eqref{eq:asympt_eigen_sobolev}. The lower bound in~\eqref{eq:asympt_eigen_sobolev} is instead just given by~\eqref{eq:lower_bound_eig}, which holds under the 
assumptions that the points are quasi-uniform.

Now obviously, if the rate in Lemma~\ref{lemma:product} could be improved for this case, then we would obtain contradicting upper and lower bounds 
in~\eqref{eq:asympt_eigen_sobolev}.
\end{proof}

\begin{remark}
The bound~\eqref{eq:unc_princ} can be understood as a sort of nonlinear form of uncertainty relation or trade-off principle (see \cite{Schaback1995, 
Schaback2023}). Indeed, it proves that good approximability of $\conv(\Phi(\Omega))$, by means of a nonlinear process, necessarily causes the kernel matrix to
have a small determinant and a small minimal eigenvalue.
Moreover, \eqref{eq:asympt_eigen_sobolev} provides an exact asymptotic of the determinant of the kernel matrix for certain kernels and point distributions, in 
a form that we are not aware has been proven in the literature. This is analogous to results used in polynomial interpolation to analyze some optimality 
property of Fekete points (see e.g.~\cite{Bloom1992,Bloom2012}), which have been considered recently also for kernel interpolation~\cite{Karvonen2020b}.
\end{remark}

We next turn to look at the optimality of our main result, and prove the following proposition using the lower bound on the power function.
\begin{proposition}
Under the hypotheses of Corollary~\ref{thm:new_conv_rates_explicit_sobolev}, assume additionally that $\ns$ is norm equivalent to $W_2^\tau(\Omega)$ 
for $\tau>d/2$.
Then Corollary~\ref{thm:new_conv_rates_explicit_sobolev} with $\beta=0$, $\gamma=1$, implies that there is $c'>0$ such that
\begin{equation*}
\norm{L_\infty(\Omega)}{P_{X_{2n}}}
\leq c' (2n)^{-\frac{\tau}{d}+\frac{1}{2}}, \;\;n\in\N.
\end{equation*}
In view of~\eqref{eq:lower_bound_pf}, it is not possible to improve the rates of Proposition~\ref{prop:entropy_as_lip}, and the rates of 
Corollary~\ref{thm:new_conv_rates_explicit_sobolev} and Theorem~\ref{thm:new_conv_rates} for $\beta=0$.
\end{proposition}
\begin{proof}
For $\beta=0$ ($P$-greedy) the selected points are independent of $f$, and thus Corollary~\ref{thm:new_conv_rates_explicit_sobolev} with $\beta=0$, 
$\gamma=1$, applies to any $f\in\ns$ with a fixed set of points $X_{2n}\subset\Omega$.
Let $x_{2n+1}\in\Omega$ be the next point selected by $P$-greedy, and let
$f\coloneqq k(\cdot, x_{2n+1}) - I_{X_{2n}}(k(\cdot, x_{2n+1}))$, so that
\begin{equation}\label{eq:beta_opt_tmp}
\norm{L_\infty(\Omega)}{P_{X_{2n}}}
=
P_{X_{2n}}(x_{2n+1})
=
\frac{\left|f(x_{2n+1})\right|}{\norm{\ns}{f}}
\leq
\frac{\norm{L_\infty(\Omega)}{f}}{\norm{\ns}{f}}.
\end{equation}
Then for all $n+1\leq i \leq 2n$ we have $f_{|X_i}=0$, and thus $I_{X_i} f=0$ and $f - I_{X_i} f=f$. In particular
\begin{equation*}
\min\limits_{n+1\leq i\leq 2n}\norm{L_\infty(\Omega)}{f - I_{X_i} f} = \norm{L_\infty(\Omega)}{f},\;\; \norm{\ns}{f - I_{X_{n+1}}f} = \norm{\ns}{f},
\end{equation*}
and thus Corollary~\ref{thm:new_conv_rates_explicit_sobolev} for $\beta=0$, $\gamma=1$, and~\eqref{eq:beta_opt_tmp} prove that
\begin{equation*}
\norm{L_\infty(\Omega)}{P_{X_{2n}}}
\leq\frac{\norm{L_\infty(\Omega)}{f}}{\norm{\ns}{f}}
\leq C \sqrt{5}\ n^{-\frac{\tau}{d}+\frac{1}{2}}
\leq c (2n)^{-\frac{\tau}{d}+\frac{1}{2}}, \;\;n\in\N,
\end{equation*}
where $c\coloneqq C \sqrt{5}\ 2^{\frac{\tau}{d}-\frac{1}{2}}$. This matches the lower bound of~\eqref{eq:lower_bound_pf}.
\end{proof}
It remains open to understand if Theorem~\ref{thm:new_conv_rates} can instead be improved for values $\beta\neq0$, e.g., if a better rate can be
obtained for target-data-dependent selection rules.
To address this point, we construct in Section~\ref{sec:bb_example} an explicit example that shows that this is not the case in general. 
However, before moving to this example, we compare the new rates proven for $f$-greedy to those of an optimal uniform nonlinear approximation.

\subsection{Comparison with optimal uniform nonlinear approximation}\label{sec:non_lin}
The problem of how well functions can be approximated in a Reproducing Kernel Hilbert Space is well studied, and there is an interesting gap occurring between 
different approximation schemes (see \cite{Santin2016a,Steinwart2017}).

Namely, \cite{Steinwart2017} considers the following types of approximation quantities, where $B_{\ns}$ is the unit ball in $\ns$. The interpolation $n$-width 
\begin{equation*}
I_n(B_{\ns}, L_\infty(\Omega))
\coloneqq \inf\limits_{X\subset \Omega, |X|\leq n} \norm{L_\infty(\Omega)}{\sup\limits_{f\in B_{\ns}}|(f - I_X f)(x)|}
= \inf\limits_{X\subset \Omega, |X|\leq n} \norm{L_\infty(\Omega)}{P_X},
\end{equation*}
measures how well any $f\in B_{\ns}$ can be approximated by interpolation with at most $n$ points. This is the rate that gives a lower bound for
the approximation via target-data-independent selection rules, and \cite{Santin2017x} proves that this lower bound is matched by $P$-greedy ($\beta=0$).

One may hope to improve this rate by moving to general linear approximation algorithms, and the best approximation error in this case is given by the 
approximation $n$-width. Given the space $L_n(\ns, L_\infty(\Omega))$ of bounded linear operators $A:\ns\to L_\infty(\Omega)$ of rank at most $n$, this is 
defined by 
\begin{equation*}
a_n(B_{\ns}, L_\infty(\Omega))\coloneqq \inf\limits_{A\in L_n(\ns, L_\infty(\Omega))} \sup\limits_{f\in B_{\ns}} \norm{L_\infty(\Omega)}{f - A f}.
\end{equation*}
A further generalization of the 
admissible approximation schemes leads to the consideration of the Kolmogorov $n$-width $d_n(\ns, L_\infty(\Omega))$, which measures how well any
$f\in B_{\ns}$ can be approximated in the $L_\infty$-sense by an arbitrary (nonlinear) approximation scheme with $n$-terms. Observe that this Kolmogorov width 
is different from the Kolmogorov width $d_n(K, \ns)$, $K\subset\ns$ considered before in this paper.

It clearly holds that
\begin{equation*}
d_n(\ns, L_\infty(\Omega))
\leq
a_n(\ns, L_\infty(\Omega))
\leq
I_n(\ns, L_\infty(\Omega)),
\end{equation*}
but one may wonder how tight these inequalities are.

The paper \cite{Steinwart2017} addresses precisely the existence of a gap between these quantities.
In particular, among other more general results \cite{Steinwart2017} proves in Corollary 5.1 that if $\ns$ is norm equivalent to $W_2^\tau(\Omega)$, 
then there are constants $c_1, c_2, c_3, c_4>0$ such that 
\begin{align}\label{eq:is_gap}
c_1 n^{-\tau/d} &\leq d_n(\ns, L_\infty(\Omega))\leq c_2 n^{-\tau/d},\\
c_3 n^{-\tau/d + 1/2} &\leq a_n(\ns, L_\infty(\Omega)) \leq I_n(\ns, L_\infty(\Omega))\leq c_4 n^{-\tau/d + 1/2},\nonumber
\end{align}
i.e., there is a gap of $n^{1/2}$ between best nonlinear uniform approximation and best uniform interpolation, and moving to general linear algorithms is not 
sufficient to close this gap.
In view of the results of Corollary~\ref{thm:new_conv_rates_explicit_sobolev}, we see that one can instead close this gap by using adaptivity (i.e., $f$-greedy, 
$\beta=1$), since indeed the convergence rate of Corollary~\ref{thm:new_conv_rates_explicit_sobolev} matches the one of $d_n(\ns, L_\infty(\Omega)$ 
in~\eqref{eq:is_gap}.

Observe, however, that $d_n(\ns, L_\infty(\Omega))$ still does not capture the best rate achievable by $f$-greedy, and simply modifying its definition to allow
an adaptive selection of the space $V_n$ would of course result in a zero approximation error (just pick $f\in V_n$). 
One may instad consider the width 
\begin{equation}\label{eq:best_adaptive_interp}
\sup\limits_{f\in \ns} \inf\limits_{X\subset \Omega, |X|\leq n} \norm{L_\infty(\Omega)}{f - I_X f},
\end{equation}
of optimal adaptive interpolation. 
We are not aware of result that quantify the decay of this width, and its rate could in principle be significantly better than those of $d_n(\ns, 
L_\infty(\Omega))$ and of $f$-greedy. However, the example presented in the next section demonstrates that this is not the case in general.

\subsection{A limiting case for \texorpdfstring{$\beta=1$}{beta equal 1}}\label{sec:bb_example}
We discuss here an example where the form of the kernel allows us to perform explicit calculations, which show that the rates of
Theorem~\ref{thm:new_conv_rates} cannot be improved if they have to be valid for a general strictly positive definite kernel, and that the rate of 
convergence of $f$-greedy may match those of optimal $L_\infty$ interpolation in the sense of~\eqref{eq:best_adaptive_interp}.

The example pertains $\Omega\coloneqq(0,1)\subset\R$ (i.e., $d=1$) and the Brownian Bridge kernel $k:\Omega\times\Omega\to\R$, $k(x,y)\coloneqq \min(x,y) -x y$.
Its native space is $\ns=H_0^1((0,1))$ (i.e., the Sobolev space $W_2^1(\Omega)$ with zero boundary conditions) with the inner product 
\begin{equation*}
\inner{\ns}{u,v} = \left(\int_0^1u'(x)v'(x) dx\right)^{1/2}, 
\end{equation*}
and this space is embedded in $C^{0,1/2}(\Omega)$. Thus, we expect from Corollary~\eqref{thm:new_conv_rates_explicit} with $s=1/2$ to have rates $n^{-1/2 -1/2} 
=
n^{-1}$ for $\beta=1$. Kernel interpolation in $\ns$ coincides with piecewise linear spline interpolation with zero boundary conditions.

Observe that for this kernel it is known (see Theorem 5 in~\cite{Santin2022b}) that 
\begin{equation*}
\inf_{X_n\subset \Omega, |X_n|\leq n} \norm{L_\infty(\Omega)}{P_{X_n}} = \frac12 (n+1)^{-1/2},
\end{equation*}
where the infimum is realized when $X_n\coloneqq\{{i}/{(n+1)},\ 1\leq i\leq n\}$ are equally spaced points.

For a certain function $f\in\ns$, we can explicitly describe the approximation provided by $f$-greedy, and we can even compare it with the $L_\infty$-optimal 
interpolation, i.e., the one obtained with points\footnote{Observe that the infimum does not need to be attained in general, but we show that an optimal $\bar 
X_n$ exists in this case.}
\begin{equation*}
\bar X_n = \argmin\limits_{X_n\subset \Omega, |X_n| = n} \norm{L_\infty(\Omega)}{f - I_{X_n} f}.
\end{equation*}
Calculations that are reported in Appendix~\ref{sec:bb_computations} show the following. 
\begin{lemma}\label{lemma:bb}
Let $f(x)\coloneqq x (1-x)\in\ns$. 
For every $n\in\N$, the set $\bar X_n\coloneqq\{{i}/{(n+1)},\ 1\leq i\leq n\}$ of $n$ equally spaced points is the unique set of $L_\infty(\Omega)$-optimal 
interpolation points for $f$, and setting $\bar r_n\coloneqq f - I_{\bar X_n} f$ we have
\begin{equation*}
\norm{L_\infty(\Omega)}{\bar r_{n-1}} = n^{-2}/4,\;\;
\norm{\ns}{\bar r_{n-1}} = n^{-1}/\sqrt{3},\;\;n\in\N.
\end{equation*}
The $f$-greedy algorithm applied to $f$ selects instead the points $X_n$ defined by bisecting the largest interval, i.e., 
\begin{equation*}
x_{2^{\ell-1} + i} = (2i + 1) / 2^{\ell},\;\; \ell\in\N, \;\; i=0,\dots, 2^{\ell-1}-1,
\end{equation*}
where the points corresponding to the same index $\ell$ may be chosen in any order.
For these points, setting  $r_n\coloneqq f - I_{X_n} f$ and $n\coloneqq 2^{\ell-1} + i$, we have
\begin{equation}\label{eq:error_f_greedy_bb}
\norm{L_\infty(\Omega)}{r_{n-1}} = 2^{-2\ell},\;\;
\norm{\ns}{r_{n-1}} =\left(\left(2^{1-3\ell}/3\right) \left(2^{\ell+1}  - 3 i\right)\right)^{1/2}, \;\; \ \ell\in\N,\  i=0,\dots,\ 2^{\ell-1}-1.
\end{equation}
In particular it holds
\begin{equation*}
\frac{\sqrt{3}}{4\sqrt{2}} (n+1)^{-1}
= \frac{1}{\sqrt{2}} \frac{\norm{L_\infty((0,1))}{\bar r_n}}{\norm{\ns}{\bar r_n}}
\leq \frac{\norm{L_\infty((0,1))}{r_n}}{\norm{\ns}{r_n}}
\leq 4 \frac{\norm{L_\infty((0,1))}{\bar r_n}}{\norm{\ns}{\bar r_n}}
= \sqrt{3} (n+1)^{-1},\;\; n\in\N.
\end{equation*}
\end{lemma}

The lemma shows that there is a kernel $k$ and a function $f\in\ns$ for which the rates of Theorem~\ref{thm:new_conv_rates} cannot be improved asymptotically 
even for
$\beta=1$, and consequently the same holds for its corollaries. Moreover, this lemma even shows that $f$-greedy points provide the same rate of 
$L_\infty$-approximation as the $L_\infty$-optimal interpolation points.
Observe that in this case we prove that the rate $s/d-1/2$ cannot be improved. However, this rate happens to coincide with $2s/d$, and one may argue that a
better rate could be obtained for smoother kernels. We show numerically in Section~\ref{sec:numerics} that this seems not to be the case.

We report in Figure~\ref{fig:bb_example} the rates obtained by numerically evaluating the errors described in Lemma~\ref{lemma:bb}, where experiments are
run up to $n=127=2^7-1$, running the $f$-greedy selection on $5001$ equally spaced points in $\Omega$, and evaluating the $L_\infty$ errors by point evaluation 
on $10001$ 
equally spaced points in $\Omega$.

Regarding the relation between $f$-greedy and optimal points, it is interesting to observe that the first are nested, and thus they are optimized by bisecting 
the largest interval, while the optimal ones are equally spaced. This implies in particular that they coincide for $n=2^\ell$. 
For the intermediate values of $n$, however, the $L_\infty$ error of the $f$-greedy points does not decrease and the $\calh$-error is instead 
monotonically decreasing (equation~\eqref{eq:error_f_greedy_bb}). This implies that there are arbitrarily long sequences of consecutive points with an increase of the ratio between the two errors.

\begin{figure}[ht]
\centering
\begin{tabular}{cc}
\includegraphics[width=0.5\textwidth]{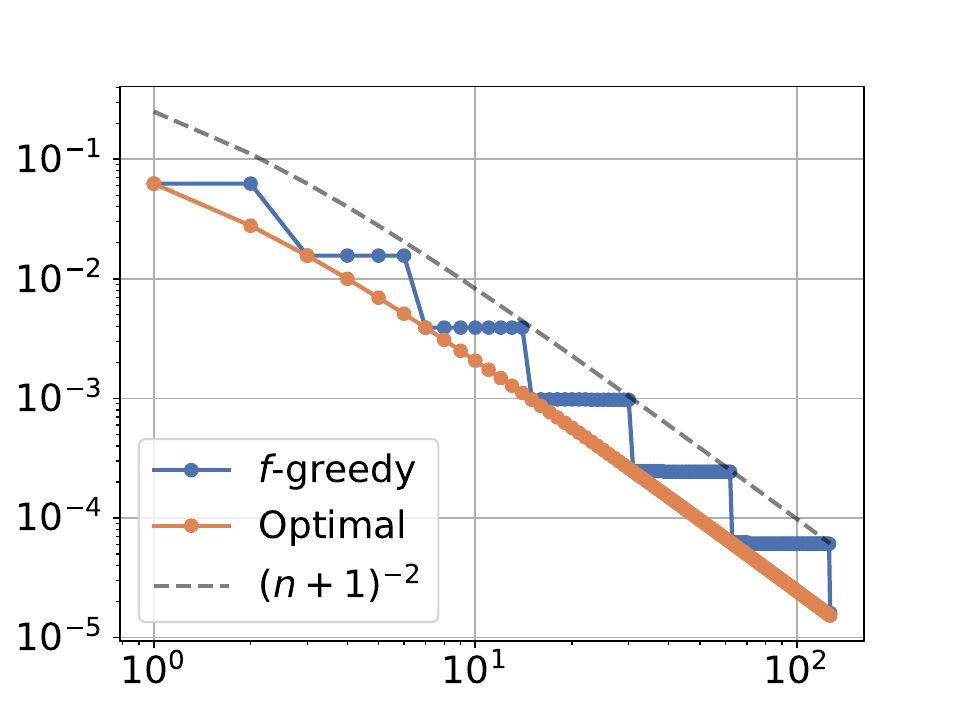}
&
\includegraphics[width=0.5\textwidth]{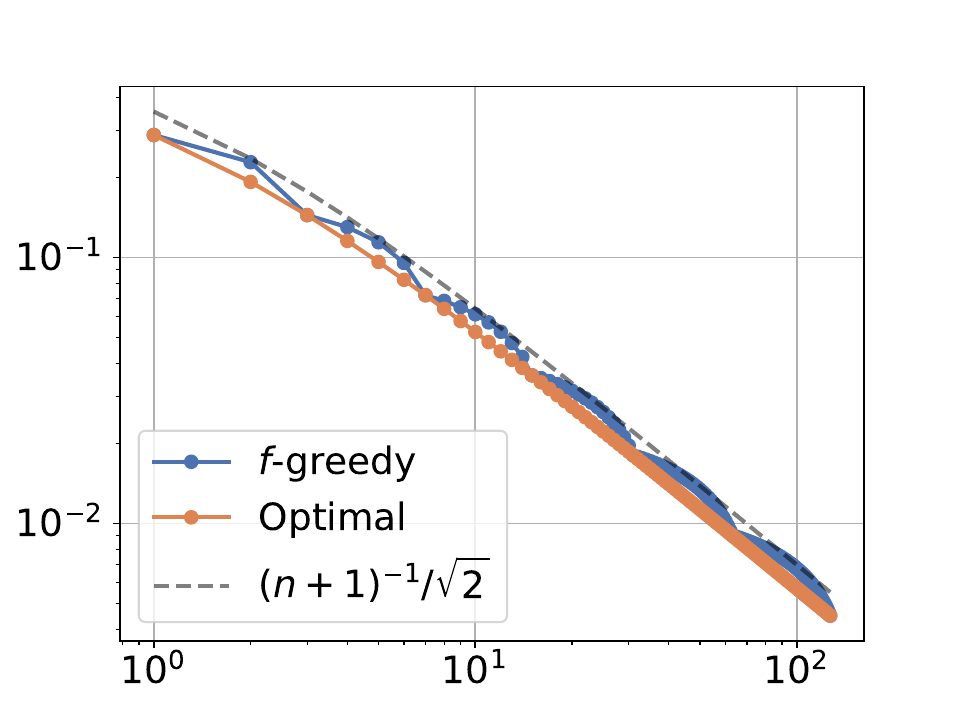}
\end{tabular}
\includegraphics[width=0.5\textwidth]{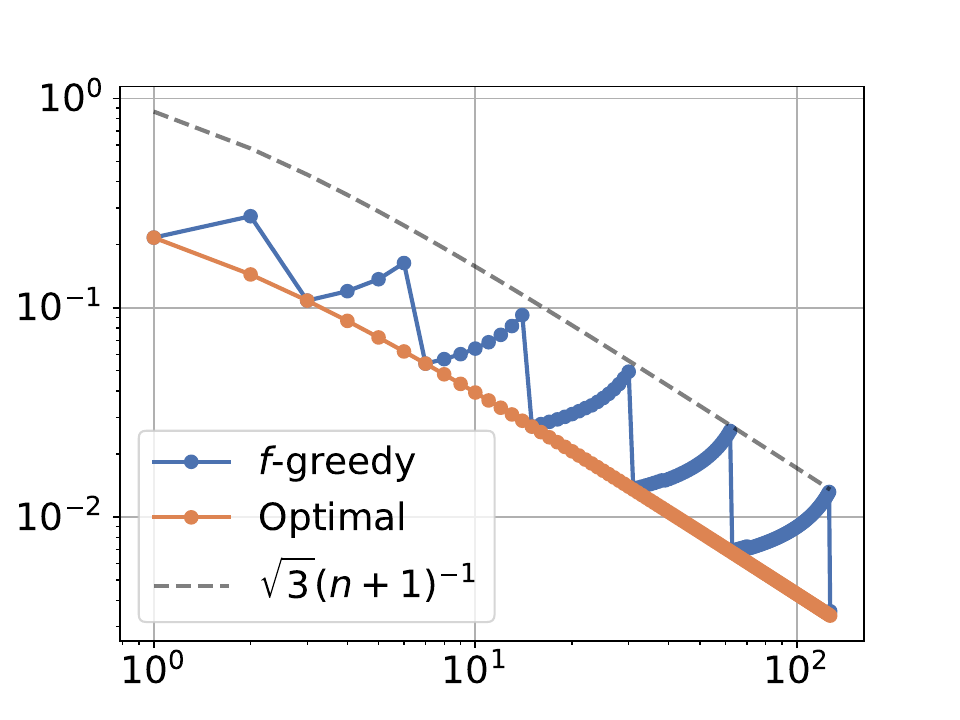}
\caption{
Decay of the error ($y$-axis) for $f\coloneqq x(1-x)$ as a function of the number of interpolation points ($x$-axis), interpolated with the Brownian Bridge kernel in $\Omega=(0,1)$ by $f$-greedy and optimal points. The plots show the
decay of the $L_\infty$-error (left), the $\ns$-error (right), and their ratio (bottom), together with their algebraic rates.
The $x$- and $y$-scales of the plot are logarithmic.}\label{fig:bb_example}
\end{figure}

Moreover, this example shows that it is not the $\ns$-norm convergence in itself to provide a faster decay in 
Corollary~\ref{thm:new_conv_rates_explicit_sobolev}, since in this example we have $\norm{\ns}{r_n}\leq c n^{-1}$, while for general functions one can only 
expect convergence of this quantity, but with an arbitrarily slow rate (see e.g. Chapter 8 in  \cite{Iske2019}). It seems instead plausible that the issue is 
rather related to adaptivity, i.e., one may get a rate faster than $n^{-1}$ when the points selected by $f$-greedy deviate significantly from the uniformly 
optimal ones (which are equally spaced for this kernel, see \cite{Santin2022b}).
We will test this fact further in the two examples in 
Section~\ref{sec:numerics}.

\section{Numerical experiments}\label{sec:numerics}
An extensive numerical investigation of $\beta$-greedy has already been discussed in \cite{Wenzel2022c}, to which we refer for further details. 

Here, we report the results of two simple experiments that target some novel aspects discussed in this paper. 
Both experiments are obtained on $\Omega\coloneqq(0,1)$ with the Brownian Bridge kernel, that is denoted in this section as $k_1$ with its native space 
$\calh_1$, and its iterated version
\begin{equation*}
k_2(x, y)\coloneqq \int_0^1 k_1(x, z) k_1(z, y)dz = -\frac{1}{6} \left(\min(x, y) - x y\right) \left(x^2 + y^2 - 2 \max(x, y)\right),
\end{equation*}
whose native space $\calh_2$ is $H_0^2(\Omega)$ (i.e., the Sobolev space $W_2^2(\Omega)$ with zero boundary conditions of second order) with inner product 
$\inner{\calh_2}{u, v} \coloneqq \int_{0}^1 u''(x) v''(x) dx$ (see \cite{Fasshauer2015}, Chapter 7). For this space we expect a rate $s/d=-3/2$ in 
Corollary~\ref{thm:new_conv_rates_explicit_sobolev}.

We consider the target function $f_p(x)\coloneqq 2^{2p} (x (1-x))^p$. It differs from the function used in Section~\ref{sec:bb_example} in the exponent $p$ and in the normalization factor, which is chosen so that $f_p(1/2) = \norm{L_\infty(\Omega)}{f_p}=1$ to simplify the comparison between different values of $p$.
The function is in the native space of $k_1$ if $p>1/2$ and of $k_2$ if $p>3/2$. In these cases, one can 
explicitly compute
\begin{equation*}
\norm{\calh_1}{f_p}^2 = \frac{2  p \sqrt\pi}{2 p - 1}\cdot \frac{\Gamma(2 p+1)}{\Gamma(2  p + 1/2)},\;\;
\norm{\calh_2}{f_p}^2 = \frac{48 (p-1) p^2 \sqrt{\pi}}{2 p-3}\cdot \frac{\Gamma(2 p-1)}{\Gamma(2 p-1/2)}.
\end{equation*}
In both experiments, this function is interpolated with $f$-greedy, which is run with the same discretization of $\Omega$ described in 
Section~\ref{sec:bb_example}, except that the iteration is stopped before $n=127$ when the selected points are too close together.

We first use $k_1$ to interpolate the functions $f_{0.51}$ and $f_{3}$. The results are reported in Figure~\ref{fig:bb_example_p_beta1},
where we show the rate of decay of the $\calh_1$ and $L_\infty$ errors, and their ratio. It is interesting to observe that for both functions the $L_\infty$ 
error decays as $n^{-2}$, but the $\calh_1$ error of $f_{3}$ decays roughly as $n^{-1}$ (as for $f_1$, analyzed in Section~\ref{sec:bb_example}), while for
$f_{0.51}$, which is close to not being in $\calh_1$, the $\calh_1$ error is almost constant. This has the effect that the ratio of the two errors for 
$f_{3}$ is close to the limiting $n^{-1}$, as in Section~\ref{sec:bb_example}, while for $f_{0.51}$ it decay almost as $n^{-2}$, i.e., faster than the
prediction of Corollary~\ref{thm:new_conv_rates_explicit_sobolev}.  
This experiment confirms that the reason for a fast decaying error in Corollary~\ref{thm:new_conv_rates_explicit_sobolev} needs not to be a fast decay of the 
$\ns$-norm of the error. As anticipated in Section~\ref{sec:bb_example}, one possible reason for this behavior is the distribution of the selected points 
(bottom right in Figure~\ref{fig:bb_example_p_beta1}): for $f_{3}$ the points are almost equally spaced, while for $p=0.51$ they are concentrated close to
the boundary. 
\begin{figure}[ht]
\centering
\begin{tabular}{cc}
\includegraphics[width=0.5\textwidth]{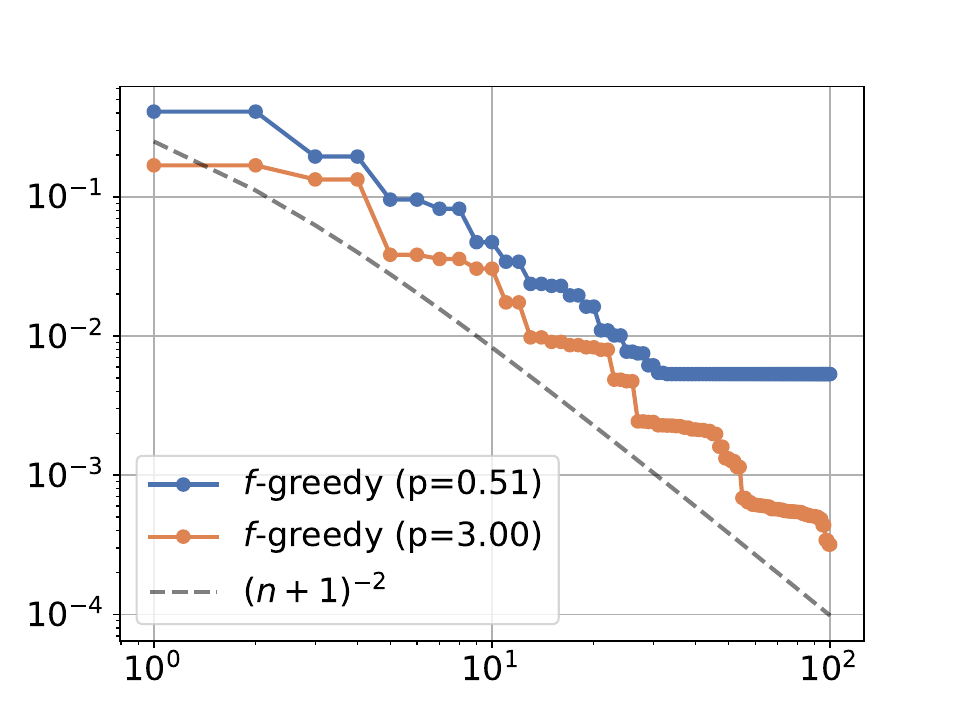}
&
\includegraphics[width=0.5\textwidth]{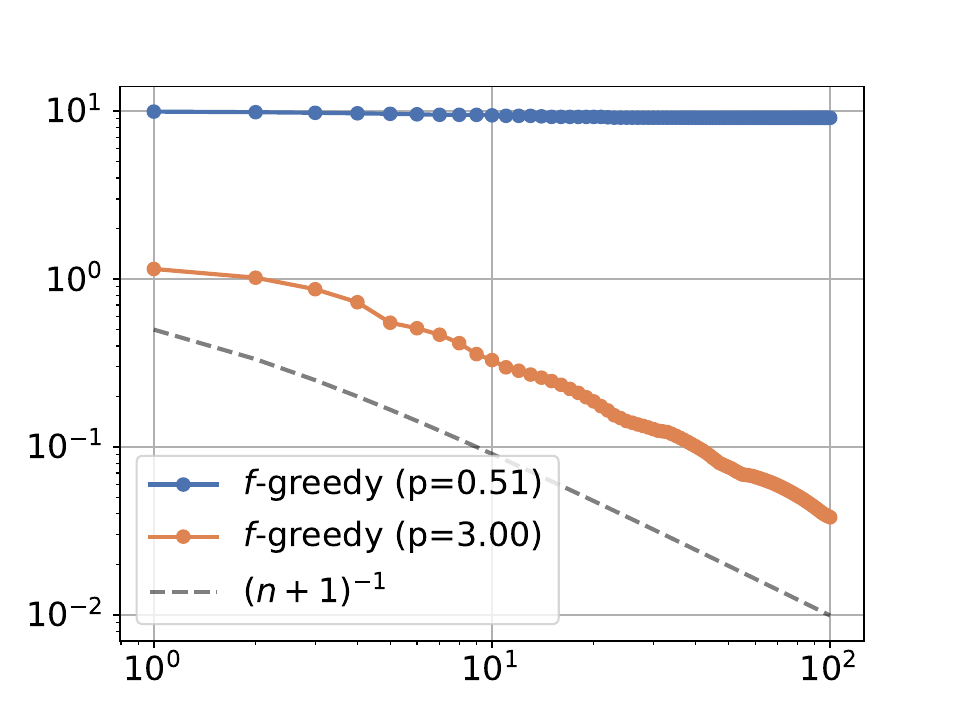}
\\
\includegraphics[width=0.5\textwidth]{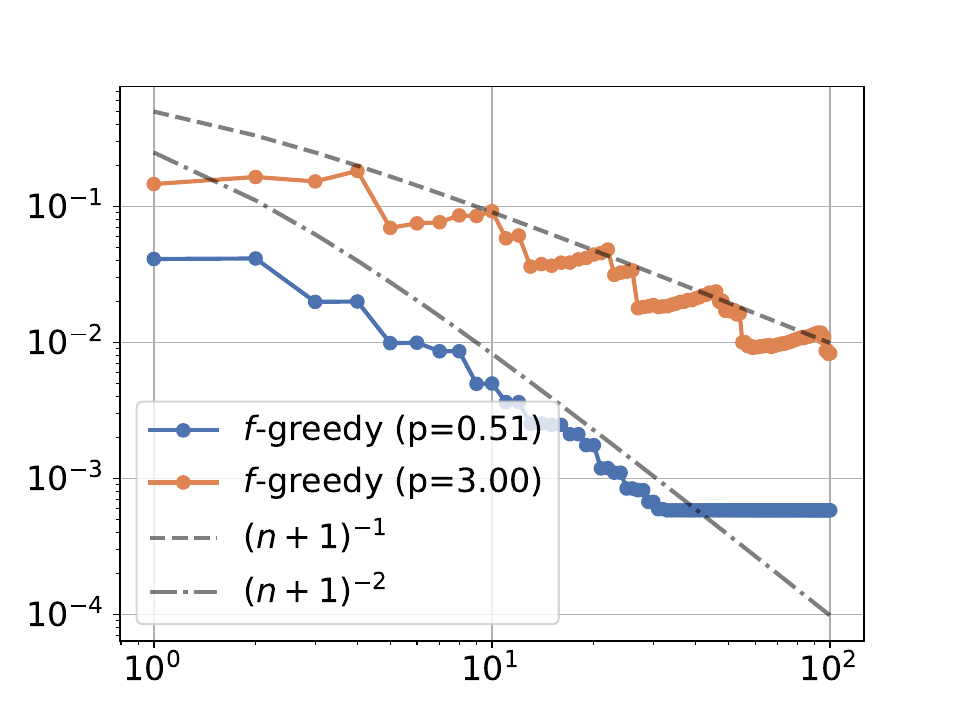}
&
\includegraphics[width=0.5\textwidth]{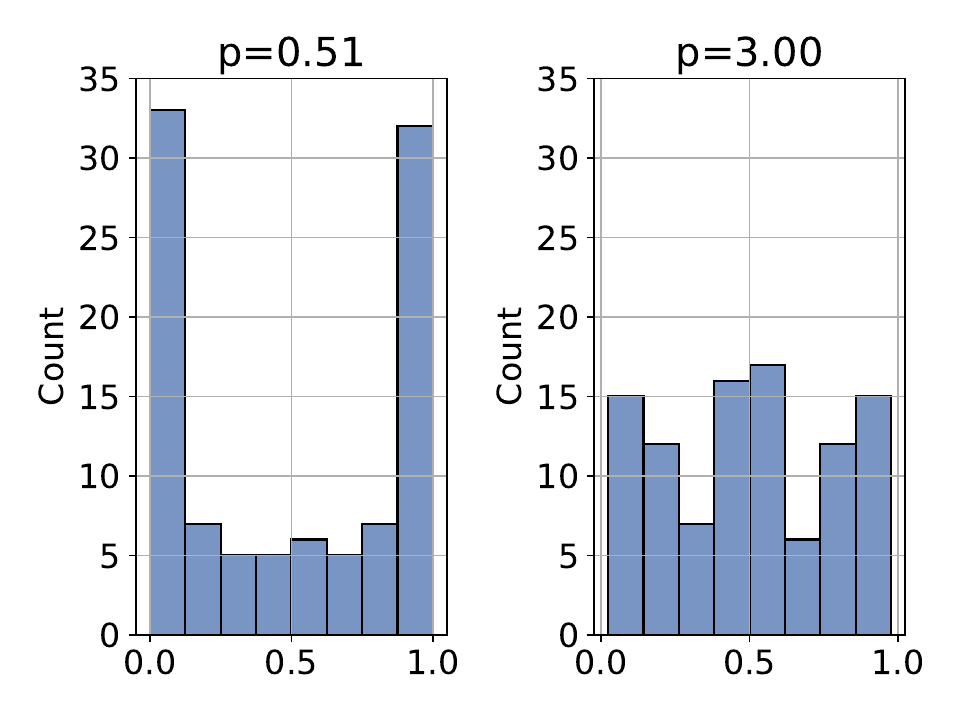}
\end{tabular}
\caption{
Results of the interpolation of $f_p$, for values $p=0.51, 3$, by $f$-greedy in $\Omega=(0,1)$ with kernel $k_1$.
The plots show the decay of the $L_\infty$-error (top left), the $\ns$-error (top right), and their ratio (bottom left), together with their estimated algebraic rates.
The bottom right figure shows a density plot of the selected interpolation points, i.e., a count of the number of points selected in each represented subinterval.
}\label{fig:bb_example_p_beta1}
\end{figure}

In the second experiment we consider instead $k_2$ and $f_{1.51}$, $f_4$. 
The results are reported in Figure~\ref{fig:bb_example_p_beta2}, with the same 
organization as in the previous example. 
The function $f_{1.51}$ is close to be not in $\calh_2$, and indeed the behavior of the error for $f_{1.51}$ reflects the case of $f_{0.51}$ for $k_1$. The 
results for $f_4$ are instead interesting because the ratio of its $L_\infty$ and $\calh_2$ interpolation errors decays close to the rate $n^{-2}$, which is 
exactly the rate $3/2+1/2=2$ predicted by Corollary~\ref{thm:new_conv_rates_explicit_sobolev} for $f$-greedy interpolation with $k_2$. As anticipated in 
Section~\ref{sec:bb_example}, this numerically confirms the optimality of the estimate even for smoother kernels.

\begin{figure}[ht]
\centering
\begin{tabular}{cc}
\includegraphics[width=0.5\textwidth]{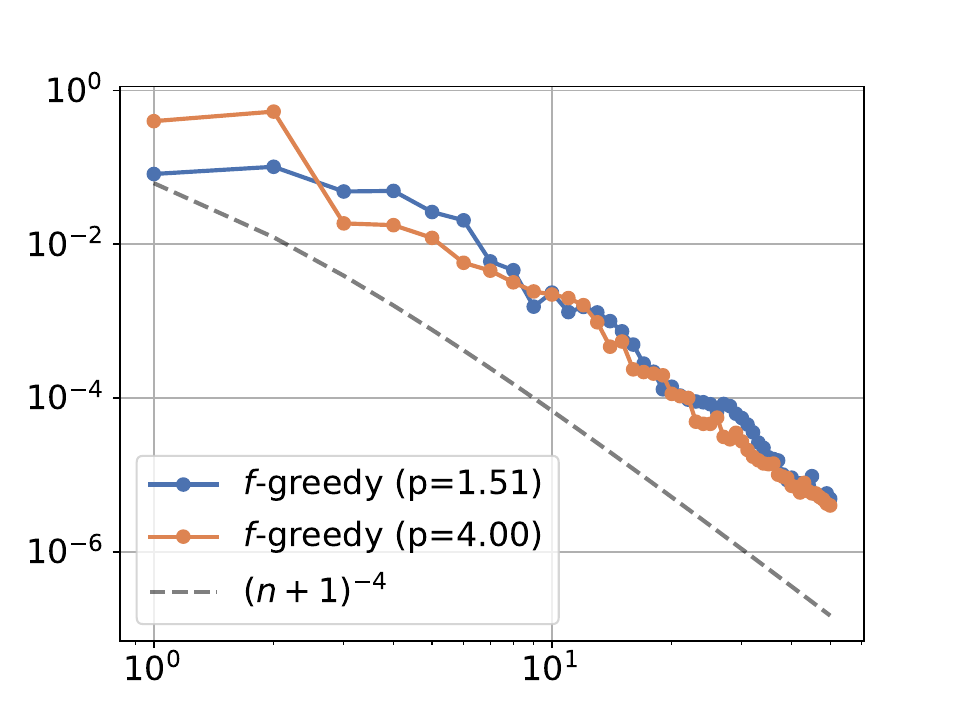}
&
\includegraphics[width=0.5\textwidth]{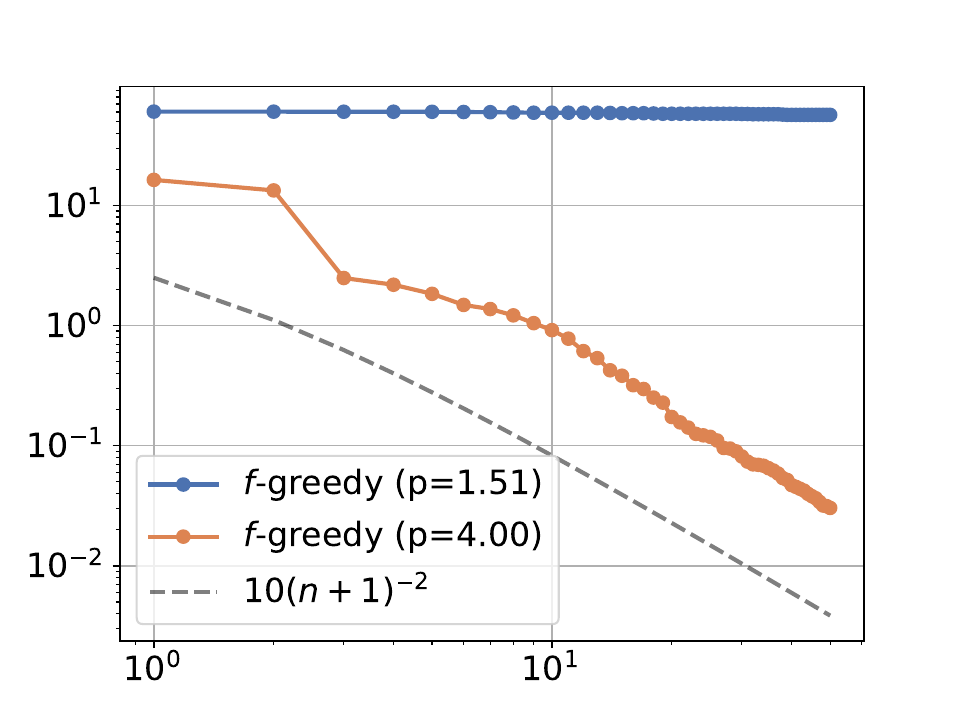}
\\
\includegraphics[width=0.5\textwidth]{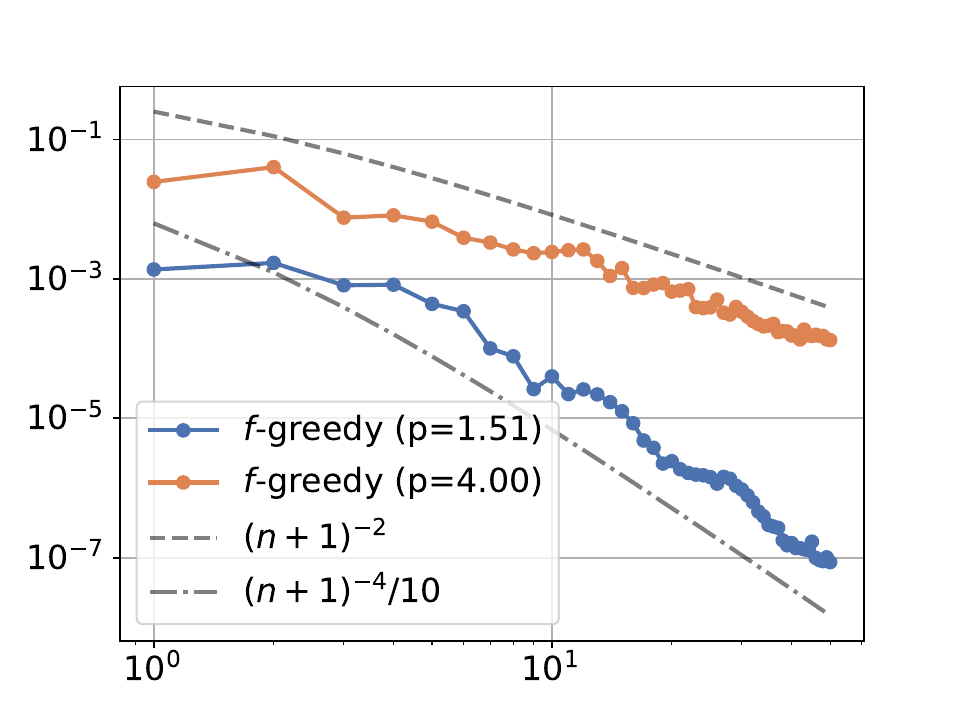}
&
\includegraphics[width=0.5\textwidth]{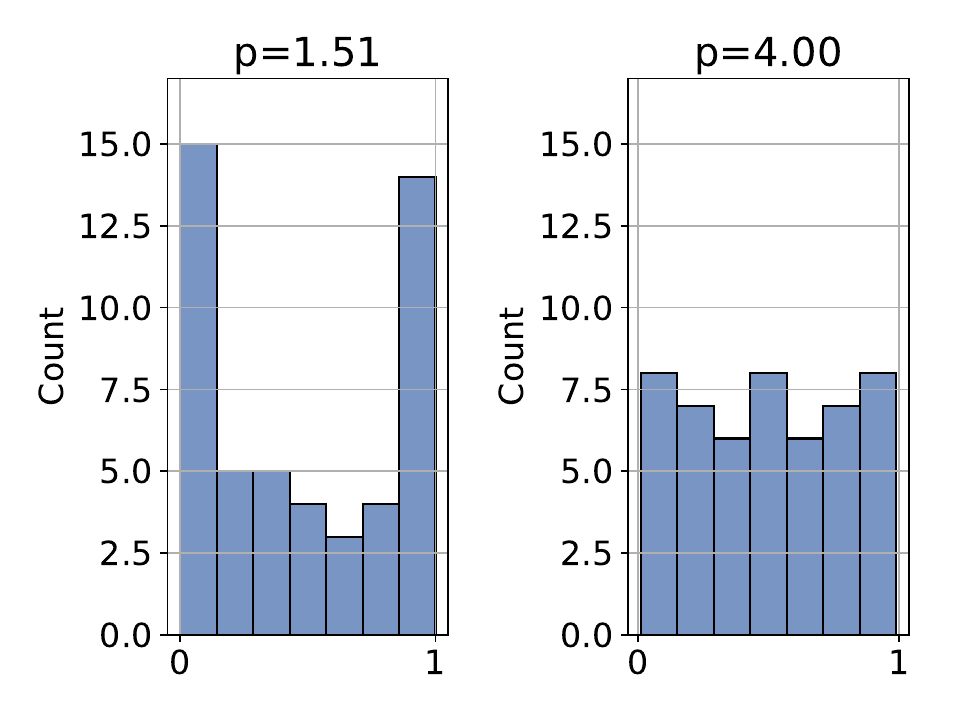}
\end{tabular}
\caption{
Results of the interpolation of $f_p$, for values $p=1.51, 4$, by $f$-greedy in $\Omega=(0,1)$ with kernel $k_2$. 
The plots show the decay of the $L_\infty$-error (top left), the $\ns$-error (top right), and their ratio (bottom left), together with their estimated 
algebraic rates. The bottom right figure shows a density plot of the selected interpolation points, i.e., a count of the number of points selected in each represented subinterval.
}\label{fig:bb_example_p_beta2}
\end{figure}

\section*{Acknowledgements}  
The second and third authors would like to thank the German Research Foundation (DFG) for support within the Germany's Excellence Strategy - EXC 2075 - 
390740016.
\bibliography{main}
\bibliographystyle{abbrv}

\appendix
\section{Additional proofs}\label{sec:appendix}
\begin{proof}[Proof of Lemma~\ref{lemma:weak_beta_greedy}]
First, the assumption that $k$ is well defined on $\bar \Omega$ ensures that $\ns\subset\calh_k(\bar \Omega)$ (see \eqref{eq:rkhs_extension}).
We split the proof into cases as follows.
\begin{enumerate}
\item For $\beta \in[0,1]$, the definition~\eqref{eq:weak_beta_greedy} gives
\begin{equation*}
\left|r_i(x)\right|^\beta P_{X_i}(x)^{1-\beta}
\leq \gamma^{-1} \left|r_i(x_{i+1})\right|^\beta P_{X_i}(x_{i+1})^{1-\beta}, \;\;x\in\overline\Omega,
\end{equation*}
which already gives the statement for $\beta=1$ by taking the supremum over $\Omega$ on the left hand side.
For any other $\beta\in(0,1]$ we can set $\bar x\coloneqq \argmax_{x\in\bar\Omega}|r_i(x)|$, and continue to
\begin{equation}\label{eq:tmp_appendix}
P_{X_i}(\bar x) 
\leq 
\gamma^{-\frac{1}{1-\beta}}\frac{|r_i(x_{i+1})|^{\frac{\beta}{1-\beta}}}{|r_i(\bar x)|^{\frac{\beta}{1-\beta}}}  P_{X_i}(x_{i+1}).
\end{equation}
Now $r_i\in\ns\subset\calh_k(\bar \Omega)$ by \eqref{eq:rkhs_extension}, and thus we can use the power function estimate 
\eqref{eq:power_fun_err_bound} written for $\calh_k(\bar \Omega)$, and insert it in \eqref{eq:tmp_appendix} to get 
\begin{align*}
\norm{L_\infty(\Omega)}{r_i} 
&
\leq \norm{L_\infty(\bar \Omega)}{r_i} 
= |r_i(\bar{x})| 
\leq P_{X_i}(\bar{x}) \norm{\calh_k(\bar\Omega)}{r_i} \\
&
\leq \gamma^{-\frac{1}{1-\beta}} \frac{|r_i(x_{i+1})|^{\frac{\beta}{1-\beta}}}{|r_i(\bar{x})|^{\frac{\beta}{1-\beta}}}  
P_{X_i}(x_{i+1})\norm{\calh_k(\bar\Omega)}{r_i}\\
&= \gamma^{-\frac{1}{1-\beta}} \frac{|r_i(x_{i+1})|^{\frac{\beta}{1-\beta}}}{\norm{L_\infty(\bar\Omega)}{r_i}^{\frac{\beta}{1-\beta}}}  P_{X_i}(x_{i+1}) 
\norm{\calh_k(\bar\Omega)}{r_i}\\
&\leq \gamma^{-\frac{1}{1-\beta}} \frac{|r_i(x_{i+1})|^{\frac{\beta}{1-\beta}}}{\norm{L_\infty(\Omega)}{r_i}^{\frac{\beta}{1-\beta}}}  P_{X_i}(x_{i+1}) 
\norm{\calh_k(\bar\Omega)}{r_i},
\end{align*}
which yields the final result by rewriting for $\norm{L_\infty(\Omega)}{r_i}$.

\item For $\beta \in (1, \infty)$, the selection criterion~\eqref{eq:weak_beta_greedy} can be rearranged to
\begin{equation*}
\left|r_i(x)\right|^\beta
\leq \gamma^{-1} \left|r_i(x_{i+1})\right|^\beta \frac{P_{X_i}(x)^{\beta-1}}{P_{X_i}(x_{i+1})^{\beta-1}}, \;\;x\in\overline\Omega\setminus X_i,
\end{equation*}
i.e.,
\begin{equation*}
\left|r_i(x)\right|
\leq \gamma^{-1/\beta} \left|r_i(x_{i+1})\right| \frac{P_{X_i}(x)^{(\beta-1)/\beta}}{P_{X_i}(x_{i+1})^{(\beta-1)/\beta}}, \;\;x\in\overline\Omega\setminus X_i,
\end{equation*}

and taking the supremum over $\Omega \setminus X_i$ gives the result, since $P_{X_i}(x)\leq \norm{L_\infty(\Omega)}{P_{X_i}}$, and $\gamma\leq 1$ implies 
$\gamma^{-1/\beta} \leq \gamma^{-1}$.

\item For $\beta = \infty$, the criterion~\eqref{eq:weak_beta_greedy} can be directly rearranged to yield 
the statement, with the same convention $1/\infty = 0$.
\end{enumerate}
\end{proof}

\section{Calculations for the example in Section~\ref{sec:bb_example}}\label{sec:bb_computations}

For $n\in\N_0$, $X_n\coloneqq\{x_1\leq x_2\leq \dots \leq x_n\}\subset(0,1)$, we set $X_n'\coloneqq \{0,1\}\cup X_n$.
It is easy to show that the kernel interpolant on $X_n$ is the piecewise linear, continuous interpolant on $X_n'$ (i.e., a linear spline with zero 
boundary conditions), see Chapter 6 and Chapter 7 in \cite{Fasshauer2015}.

We consider the target function $f(x)\coloneqq x(1-x)\in\ns$, which has norm $\norm{\ns}{f} = 1/\sqrt{3}$.

Let $X_n$, $n\in\N_0$, with $X_0\coloneqq\emptyset$, and $X_n'\coloneqq \{0,1\}\cup X_n$, where we set $x_0\coloneqq 0$ and $x_{n+1}\coloneqq1$.
For all consecutive $a,b\in X_n'$ with $a<b$, the interpolant $I_{X_n}f$ can be written for $x\in[a,b]$ as
\begin{equation*}
I_{X_n}(x) = f(a) + \frac{f(b)-f(a)}{b-a} (x-a) = a b + (1 - a - b) x,
\end{equation*}
and thus $\left|r_n(x)\right| = r_n(x) = f(x) - I_{X_n}(x) = -(x-a)(x-b)$ is uniquely maximized in $[a,b]$ at the point $\bar x\coloneqq (a+b)/2$, where it has a value
\begin{equation}\label{eq:bb_residual}
\left|r_n(\bar x)\right|= \frac{(b-a)^2}{4}.
\end{equation}
Moreover, $r_n(x)' =  a + b - 2x$ for $x\in[a,b]$, and thus 
\begin{equation*}
\int_a^b \left(r_n(x)'\right)^2 dx
=\int_{a}^{b} (a+b-2 x)^2 dx
=\frac{(b - a)^3}{3}.
\end{equation*}
This implies that
\begin{equation}\label{eq:bb_norm}
\norm{\ns}{r_n}^2 
= \int_0^1 \left(r_n(x)'\right)^2 dx
% =\sum_{i=1}^{n+1} \int_{x_{i-1}}^{x_i} (x_{i}+x_{i-1}-2 x)^2 dx
=\sum_{i=1}^{n+1}\frac{(x_{i}-x_{i-1})^3}{3}.
\end{equation}
The relation~\eqref{eq:bb_residual} first implies that 
\begin{equation*}
\norm{L_\infty((0,1))}{f-I_{X_n}f} = \max\limits_{i=1, \dots,n+1}\frac{(x_i-x_{i-1})^2}{4},
\end{equation*}
and thus the equally spaced points $\bar X_n\coloneqq\{{i}/{(n+1)},\ 1\leq i\leq n\}$ are the unique optimal interpolation points for $f$ with respect to
the $L_\infty$ norm, for all $n\in\N$, with 
\begin{equation}\label{eq:bb_computations_tmp_three}
\norm{L_\infty((0,1))}{\bar r_n} = \frac{1}{4} (n+1)^{-2}.
\end{equation}
Moreover, using \eqref{eq:bb_norm} the residual has $\ns$-norm
\begin{equation}\label{eq:bb_computations_tmp_five}
\norm{\ns}{\bar r_n}^2
= \frac{1}{3}\sum_{i=1}^{n+1}(n+1)^{-3}
= \frac{1}{3}(n+1)^{-2},
\end{equation}
i.e., $\norm{\ns}{\bar r_n}= (n+1)^{-1}/\sqrt{3}$. Finally,
\begin{equation}\label{eq:bb_computations_tmp_six}
\frac{\norm{L_\infty((0,1))}{\bar r_n}}{\norm{\ns}{\bar r_n}} = \frac{\sqrt{3}}{4} (n+1)^{-1}.
\end{equation}

Moving to $f$-greedy, equation~\eqref{eq:bb_residual} proves as well that $f$-greedy needs to select at each iteration the midpoint of the largest interval 
that does not contain any point. This means that $f$-greedy selects points
\begin{align*}
&x_1 = 1/2\\
&x_2 = 1/4, \;\; x_3 = 3/4\\
&x_4 = 1/8, \;\; x_5 = 3/8, \;\; x_6 = 5/8, \;\; x_7 = 7/8\\
&\dots
\end{align*}
i.e., for $\ell\in\N$ the points are selected as
\begin{equation}\label{eq:bb_computations_tmp_two}
x_{2^{\ell-1} + i} = (2i + 1) / 2^{\ell},\;\; i=0,\dots, 2^{\ell-1}-1,
\end{equation}
where the points corresponding to the same $\ell\in\N$ may be selected in any order. In particular, for the final index $i=2^{\ell-1}-1$, i.e., $n = 2^{\ell-1} + i = 2^\ell-1$, we get equally spaced points.
By definition of $f$-greedy, the points~\eqref{eq:bb_computations_tmp_two} give
\begin{equation*}
\norm{L_\infty(\Omega)}{r_{n-1}} 
= |r_{n-1}(x_{n})|,
\end{equation*}
i.e., using~\eqref{eq:bb_residual}, and the fact that for fixed $\ell$ each point $x_{2^{\ell-1} + i}$ is selected in an interval with width $1/2^{\ell-1}$, we have
\begin{align*}
&\norm{L_\infty(\Omega)}{r_{0}}  = 1^2/4=2^{-2}\\
&\norm{L_\infty(\Omega)}{r_{1}} = \norm{L_\infty(\Omega)}{r_{2}}=(1/2)^2/4 = 1/16=2^{-4}\\
&\norm{L_\infty(\Omega)}{r_{3}} = \norm{L_\infty(\Omega)}{r_{4}}=\norm{L_\infty(\Omega)}{r_{5}} = \norm{L_\infty(\Omega)}{r_{6}} = (1/4)^2/4=2^{-6}\\
&\dots
\end{align*}
i.e.,
\begin{equation}\label{eq:bb_computations_tmp_one}
\norm{L_\infty(\Omega)}{r_{2^{\ell-1} + i-1}}
= \left|r_{2^{\ell-1} + i-1}(x_{2^{\ell-1} + i})\right|
=(1/2^{\ell-1})^2/4 = 2^{-2\ell},\;\; i=0,\dots, 2^{\ell-1}-1.
\end{equation}

We now compare this error with $\norm{L_\infty(\Omega)}{\bar r_{n-1}}$ (see \eqref{eq:bb_computations_tmp_three}).
As a function of $n=2^{\ell-1}+i$, the ratio between~\eqref{eq:bb_computations_tmp_one} and $\norm{L_\infty(\Omega)}{\bar r_{n-1}}$ is minimal for $i=0$ 
(equally spaced points) and maximal for $i=2^{\ell-1}-1$.
In the first case ($i=0$ and thus $n=2^{\ell-1}$) \eqref{eq:bb_computations_tmp_one} gives $\norm{L_\infty(\Omega)}{r_{n-1}}= 2^{-2\ell}=(2^{\ell-1})^{-2}/4= n^{-2}/ 4=\norm{L_\infty(\Omega)}{\bar r_{n-1}}$.
In the second case ($i=2^{\ell-1}-1$ and thus $n=2^{\ell}-1$) \eqref{eq:bb_computations_tmp_one} gives $\norm{L_\infty(\Omega)}{r_{n-1}} =
2^{-2\ell}=(n+1)^{-2}\leq 4 n^{-2}/4 = 4 \norm{L_\infty(\Omega)}{\bar r_{n-1}}$. It follows that for all $n\in\N$ we have
\begin{equation}\label{eq:bb_computations_tmp_four}
\norm{L_\infty(\Omega)}{\bar r_{n}} \leq \norm{L_\infty(\Omega)}{r_{n}} \leq 4 \norm{L_\infty(\Omega)}{\bar r_{n}}.
\end{equation}

For the $\ns$-norm of $r_n$ we compute the squared norm increment $c_n^2
= \norm{\ns}{I_{X_n}}^2-\norm{\ns}{I_{X_{n-1}}}^2$.  From \eqref{eq:bb_norm}, if $x_{n}$ is added as the midpoint of $[a,b]$, we replace a term $(b-a)^3/3$ 
with two terms $((b-a)/2)^3/3$, i.e., we have a reduction $(b-a)^3/3 - 2((b-a)/2)^3/3 = (b-a)^3/4$ in the squared $\ns$ norm.
Using again the fact that $b-a=2^{\ell-1}$ for any $\ell\in\N$, it follows that
\begin{align*}
&c_1^2  = 1^3/4=2^{-2}\\
&c_2^2 = c_3^2=(1/2)^3/4 = 2^{-5}\\
&c_4^2 = \dots = c_7^2= (1/4)^3/4=2^{-8}\\
&\dots\\
&c_{2^{\ell-1} + i}^2 = (1/2^{\ell-1})^3/4 = 2^{-3\ell+1},\;\; i=0,\dots, 2^{\ell-1}-1.
\end{align*}
To derive the decay of the $\ns$-error, we first consider the case $n=2^{\ell-1}$, for which we have
\begin{align*}
\norm{\ns}{r_{n-1}}^2
&=\norm{\ns}{r_{2^{\ell-1}-1}}^2
= \sum_{j={2^{\ell-1}}}^\infty c_j^2
= \sum_{m=\ell}^\infty \sum_{h=0}^{2^{m-1}-1} c_{2^{m-1}+h}^2
= \sum_{m=\ell}^\infty \sum_{h=0}^{2^{m-1}-1} 2^{-3m+1}\\
&= \sum_{m=\ell}^\infty 2^{m-1} 2^{-3m+1}
= \sum_{m=\ell}^\infty 2^{-2m} = \frac13 \left(2^{\ell-1}\right)^{-2} = \frac13 n^{-2},
\end{align*}
and this indeed coincides with the optimal case~\eqref{eq:bb_computations_tmp_five}, as expected since the points are equally spaced.
For the other cases $n=2^{\ell-1}+i$, $i=1, \dots, 2^{\ell-1}-1$, we have
\begin{equation*}
\norm{\ns}{r_{n-1}}^2
=\norm{\ns}{r_{2^{\ell-1}-1+i}}^2
= \norm{\ns}{r_{2^{\ell-1}-1}}^2 - \sum_{j=1}^i c_{2^{\ell-1}-1+j}^2
=\frac13 \left(2^{\ell-1}\right)^{-2} - 2^{-3\ell +1} i.
\end{equation*}
All together, for all $n=2^{\ell-1}+i$, $i=0, \dots, 2^{\ell-1}-1$, we have
\begin{equation*}
\norm{\ns}{r_{n-1}}^2
=\frac13 \left(2^{\ell-1}\right)^{-2} - 2^{-3\ell +1}i
=\frac{2^{1-3\ell}}{3} \left(2^{\ell+1}  - 3 i\right).
\end{equation*}
To compare this with the optimal error~\eqref{eq:bb_computations_tmp_five} we consider, for $n=2^{\ell-1}+i$, $\ell\in\N$, $i=0,\dots, 2^{\ell-1}-1$, the ratio
\begin{equation*}
R(i,\ell)
\coloneqq\frac{\norm{\ns}{r_{n-1}}^2}{\norm{\ns}{\bar r_{n-1}}^2}
=\frac{\frac{2^{1-3\ell}}{3} \left(2^{\ell+1}  - 3 i\right)}{\frac13 \left(2^{\ell-1}+i\right)^{-2}}
= 2^{1-3\ell}\left(2^{\ell-1}+i\right)^{2} \left(2^{\ell+1} - 3i\right).
\end{equation*}
We have
\begin{equation*}
\partial_i R(i,\ell)
= 2^{1-3\ell}\left(2\left(2^{\ell-1}+i\right) \left(2^{\ell+1} - 3i\right) -3 \left(2^{\ell-1}+i\right)^{2}\right)
= 2^{1-3\ell}\left(2^{\ell-1}+i\right) \left(5\cdot 2^{\ell-1} - 9i\right),
\end{equation*}
which vanishes at $i_1\coloneqq -2^{\ell-1}\leq0$ and $i_2\coloneqq (5/9) 2^{\ell-1}$, and it is positive in $(i_1, i_2)$. In particular, for all $\ell\in\N$ the minima of $R(i, \ell)$ for $i=0, \dots, 2^{\ell-1}-1$ are at the extrema, with
\begin{equation*}
R(0, \ell) = 2^{1-3\ell}2^{2\ell-2}2^{\ell+1}=1,
\quad\quad
R(2^{\ell-1}-1, \ell)
\geq R(2^{\ell-1}, \ell)
= 2^{1-3\ell}2^{4\ell}2^{\ell-1}=1,
\end{equation*}
and indeed one can see from \eqref{eq:bb_norm} that equally spaced points are optimal also in the $\ns$-norm. The maximum of $R(i, \ell)$ is instead at $i=i_2$, 
where
\begin{equation*}
R(i_2, \ell)
=2^{1-3\ell}(14/9)^2 2^{2\ell-2} (7/3) 2^{\ell-1}
=7^3 / 3^5<2.
\end{equation*}
It follows that
\begin{equation*}
\norm{\ns}{\bar r_{n-1}}^2
\leq
\norm{\ns}{r_{n-1}}^2
\leq 2 \norm{\ns}{\bar r_{n-1}}^2,
\end{equation*}
and combining this with~\eqref{eq:bb_computations_tmp_four} and~\eqref{eq:bb_computations_tmp_six} we get finally
\begin{equation*}
\frac{\sqrt{3}}{4\sqrt{2}} (n+1)^{-1}
=\frac{1}{\sqrt{2}} \frac{\norm{L_\infty((0,1))}{\bar r_n}}{\norm{\ns}{\bar r_n}}
\leq \frac{\norm{L_\infty((0,1))}{r_n}}{\norm{\ns}{r_n}}
\leq 4 \frac{\norm{L_\infty((0,1))}{\bar r_n}}{\norm{\ns}{\bar r_n}}
= \sqrt{3} (n+1)^{-1}.
\end{equation*}
Observe that sharper upper and lower bounds can be obtained in this last chain of inequalities by considering the ratio 
${\norm{L_\infty((0,1))}{r_n}}/{\norm{\ns}{r_n}}$ instead of its single components.
\end{document}